\documentclass[reqno,12pt,a4paper]{amsart}
\usepackage{amssymb,amscd}
\usepackage{amsmath}
\usepackage[top=3cm,bottom=3cm,left=2.5cm,right=2.5cm,footskip=17mm]{geometry}
\usepackage[mathcal]{eucal}
\usepackage{array,float}
\usepackage{enumitem}
\usepackage[T1]{fontenc}
\usepackage[USenglish]{babel}
\usepackage{chngcntr}
\usepackage{apptools}
\AtAppendix{\counterwithin{theorem}{section}}
\usepackage{float}
\newfloat{Diagram}{htbp}{dia}
\usepackage{silence}
\usepackage{amsmath}
\usepackage{amssymb} 

\title{Invariants of $r$-spin TQFTs and non-semisimplicity}
		
\author{Nils Carqueville \hspace{1em} Ehud Meir \hspace{1em} L\'or\'ant Szegedy}
\usepackage{hyperref}
\hypersetup{
  pdftitle={Invariants of $r$-spin TQFTs and non-semisimplicity -- N. Carqueville, E. Meir, L. Szegedy},
  pdfauthor={N. Carqueville, E. Meir, L. Szegedy}
}

\usepackage{xy}
\input xy
\xyoption{all}

\usepackage{pdflscape}
\usepackage[draft]{graphicx}
\usepackage{tikz-cd}
\usepackage{tikz,pgf}
\usetikzlibrary{calc}
\usetikzlibrary{decorations.markings}
\usetikzlibrary{fadings,decorations.pathreplacing}
\usetikzlibrary{matrix,arrows}
\usetikzlibrary{patterns}
\usetikzlibrary{arrows,calc,decorations.pathreplacing,decorations.markings,shapes.geometric,shadows}
\usepackage{tikzit}
\tikzstyle{1function}=[fill=white, draw=black, shape=rectangle, minimum width=0.75cm, minimum height=1 cm]
\tikzstyle{2function}=[fill=white, draw=black, shape=rectangle, minimum width=1cm, minimum height=1 cm]
\tikzstyle{3function}=[fill=white, draw=black, shape=rectangle, minimum width=1.5cm, minimum height=1cm]
\tikzstyle{multi function}=[fill=white, draw=black, shape=rectangle, minimum width=5cm, minimum height=1cm]
\tikzstyle{multi function small}=[fill=white, draw=black, shape=rectangle, minimum width=4cm, minimum height=1cm]
\tikzstyle{Multi function smaller}=[fill=white, draw=black, shape=rectangle, minimum width=3.5 cm, minimum height=1 cm]

\tikzstyle{thickedge}=[-, line width=2.5mm]

\usepackage{xcolor}
\usepackage{color}
\numberwithin{equation}{section}

\theoremstyle{definition}
\newtheorem{theorem}{Theorem}[section]
\newtheorem{lemma}[theorem]{Lemma}
\newtheorem{proposition}[theorem]{Proposition}
\newtheorem{corollary}[theorem]{Corollary}

\newtheorem{definition}[theorem]{Definition}

\newtheorem{example}[theorem]{Example}

\newtheorem{remark}[theorem]{Remark}

\newtheorem{notation}[theorem]{Notation}

\newcommand{\ot}{\otimes}
\newcommand{\C}{\mathcal{C}}
\newcommand{\N}{\mathbb{N}}
\newcommand{\lra}{\xrightarrow}
\newcommand{\one}{\textbf{1}}

\newcommand{\Z}{\mathbb{Z}}
\newcommand{\Zb}{\mathbb{Z}}

\newcommand{\Ker}{\operatorname{Ker}}

\newcommand{\End}{\operatorname{End}}
\newcommand{\Hom}{\operatorname{Hom}}

\newcommand{\Rb}{\mathbb{R}}

\newcommand{\Tr}{\operatorname{tr}}

\newcommand{\T}{\mathcal{T}}

\newcommand{\winf}{\mathbb{K}[X]_{\textrm{aug}}}

\newcommand{\Rep}{\text{Rep}}

\newcommand{\D}{\mathcal{D}}
\newcommand{\Nn}{\mathcal{N}}

\newcommand{\otl}{\underline{\ot}}
\renewcommand{\epsilon}{\varepsilon}

\newcommand{\Bord}{\operatorname{Bord}}
\newcommand{\Bordr}{\operatorname{Bord}^r_2}
\newcommand{\Vect}{\operatorname{Vect}}
\newcommand{\SVect}{\operatorname{SVect}}
\renewcommand{\k}{\mathbb{K}}

\def\lmt{\longmapsto}
\def\lra{\longrightarrow}

\renewcommand{\mod}[1]{\ \operatorname{mod}\ #1}
\newcommand{\Spin}{\operatorname{Spin}}
\newcommand{\SO}{\operatorname{SO}}
\newcommand{\Zc}{\mathcal{Z}}

\renewcommand{\le}{\leqslant}
\renewcommand{\geq}{\geqslant}
\renewcommand{\leq}{\leqslant}
\renewcommand{\to}{\longrightarrow}
\renewcommand{\mapsto}{\longmapsto}
\newcommand{\E}{\mathcal{E}}
\newcommand{\bigboxtimes}{\mathop{\vcenter{\hbox{\scalebox{2.00}{$\boxtimes$}}}}\limits}
\tikzset{
	string/.style={draw=#1, postaction={decorate}, decoration={markings,mark=at position .51 with {\arrow[draw=#1]{>}}}},
	costring/.style={draw=#1, postaction={decorate}, decoration={markings,mark=at position .51 with {\arrow[draw=#1]{<}}}},
	ostring/.style={draw=#1, postaction={decorate}, decoration={markings,mark=at position .47 with {\arrow[draw=#1]{>}}}},
	ustring/.style={draw=#1, postaction={decorate}, decoration={markings,mark=at position .56 with {\arrow[draw=#1]{>}}}},
	oostring/.style={draw=#1, postaction={decorate}, decoration={markings,mark=at position .43 with {\arrow[draw=#1]{>}}}},
	uustring/.style={draw=#1, postaction={decorate}, decoration={markings,mark=at position .59 with {\arrow[draw=#1]{>}}}},
	directed/.style={string=blue!50!black}, 
	odirected/.style={ostring=blue!50!black}, 
	udirected/.style={ustring=blue!50!black}, 
	oodirected/.style={oostring=blue!50!black}, 
	uudirected/.style={uustring=blue!50!black},     
	redirected/.style={costring= blue!50!black},
	redirectedgreen/.style={costring= green!50!black},
	directedgreen/.style={string= green!50!black},
}

\tikzset{-dot-/.style={decoration={
			markings,
			mark=at position 0.5 with {\fill circle (2pt);}},postaction={decorate}}}

\tikzset{
	Fdot/.style={circle, draw, fill, inner sep=0pt}, 
	Odot/.style={circle, draw, inner sep=0.1pt, minimum size=0.1cm}
}

\address{Universit\"at Wien, Fakult\"at f\"ur Physik, Boltzmanngasse 5, 1090 Wien, \"{O}sterreich}
	\email{nils.carqueville@univie.ac.at}

\address{Institute of Mathematics, University of Aberdeen, Fraser Noble Building, Aberdeen AB24 3UE, UK}
	\email{meirehud@gmail.com}

\address{Universit\"at Wien, Fakult\"at f\"ur Physik, Boltzmanngasse 5, 1090 Wien, \"{O}sterreich}
\email{lorant.szegedy@univie.ac.at}

\makeatletter
\g@addto@macro\bfseries{\boldmath}
\makeatother
\begin{document}

\maketitle

\begin{abstract}
	For a positive integer~$r$, an $r$-spin topological quantum field theory is a 2-dimensional TQFT with tangential structure given by the $r$-fold cover of $\SO_2$. 
	In particular, such a TQFT assigns 
	a scalar invariant to every closed $r$-spin surface~$\Sigma$. 
	Given a sequence of scalars indexed by the set of diffeomorphism classes of all such~$\Sigma$, we construct a symmetric monoidal category~$\mathcal C$ and a $\mathcal C$-valued $r$-spin TQFT which reproduces the given sequence. 
	We also determine when such a sequence arises from a TQFT valued in an abelian category with finite-dimensional Hom spaces. 
	In particular, we construct TQFTs with values in super vector spaces that can distinguish all diffeomorphism classes of $r$-spin surfaces, and we show that the associated algebras are necessarily non-semisimple. 
\end{abstract}

\vfill 

\tableofcontents

\newpage

\section{Introduction}
\label{sec:Introduction}

A basic classification result in topological quantum field theory (TQFT) states that 2-dimensional oriented \textsl{closed} TQFTs are equivalent to commutative Frobenius algebras. 
Similarly explicit characterisations of oriented closed TQFTs are available also in dimensions~1 and~3, but not in higher dimension. 
The cobordism hypothesis offers a general classification of \textsl{fully extended} TQFTs in arbitrary dimension and for arbitrary tangential structures, while there are relatively few classification results for closed TQFTs with tangential structures other than orientations. 
In the present paper we are concerned with one of the exceptions in dimension~2, namely \textsl{$r$-spin TQFTs}. 

Recall that for a positive integer~$r$, the $r$-spin group $\Spin^r_2$ is the $r$-fold cover of the rotation group $\SO_2$, and that an $r$-spin surface is a surface equipped with a $\Spin^r_2$-principal bundle together with a bundle map to the oriented orthonormal frame bundle. 
There is a symmetric monoidal category $\Bord^r_2$ whose morphisms are equivalence classes of $r$-spin bordisms, which for $r=1$ recovers the familiar category of 2-dimensional oriented bordisms (see Section~\ref{sec:rspin} for more details). 
An $r$-spin TQFT with values in a given symmetric monoidal category~$\C$ is by definition a symmetric monoidal functor $\Bord^r_2 \lra \C$. 
As shown in \cite{Szegedy-Stern}, such TQFTs are classified by ``closed $\Lambda_r$-Frobenius algebras'' in~$\C$, which we review in Definition~\ref{def:ClosedRFrobeniusAlgebra} below. 

A closed $r$-spin surface~$\Sigma$ can be viewed as a morphism $\varnothing\to \varnothing$ in $\Bord^r_2$. 
Evaluating an $r$-spin TQFT $\Zc\colon\Bord^r_2\to \C$ on~$\Sigma$ produces the invariant $\Zc(\Sigma)\in \End_{\C}(\one)$. 
If moreover $\End_{\C}(\one)\cong \k$ for some field~$\k$ of characteristic~0, then $\Zc(\Sigma)$ gives a scalar invariant of~$\Zc$. 
Hence in this case, every TQFT $\Zc\colon\Bord^r_2\to \C$ gives rise to a sequence of scalars $(\Zc(\Sigma))_\Sigma$, where~$\Sigma$ runs over all diffeomorphism classes of $r$-spin surfaces. 

\medskip 

The main question that will be addressed in this paper is what kind of sequences $(\Zc(\Sigma))_\Sigma \subset \k$ arise from different $r$-spin TQFTs. 
Before we give an answer in Theorem~\ref{thm:main} below, we consider the more specific question of whether there are $r$-spin TQFTs with algebraic targets~$\C$ that can distinguish all morphisms in $\Bord_2^r$. %\footnote{Clearly for $\C = \Bord^r_2$, the identity functor trivially distinguishes all morphisms. In the present context by ``algebraic'' we mean ``$\k$-linear''.} 
In particular, it turns out that $\C = \SVect_{\k}$ is good enough for the second question. 
The choice $\C = \Vect_{\k}$ is not quite sufficient, but the only \textsl{super} algebra~$A$ needed is the one related to the Arf invariant (see Section~\ref{sec:Ex_A}). 
Indeed, in Section~\ref{sec:Ex_dist} we prove our first main result: 

\begin{theorem}
	\label{thm:semimain}
	For every integer~$r$ there are non-semisimple closed $\Lambda_r$-Frobenius algebras~$B$ (for~$r$ odd) and~$C$ (for~$r$ even) in $\Vect_{\k}$, constructed in Sections~\ref{sec:Ex_B} and~\ref{sec:Ex_C}. 
	Direct sums and tensor products of these algebras with the Arf algebra~$A$ correspond to $r$-spin TQFTs valued in $\SVect_{\k}$ that have distinct invariants for all $r$-spin surfaces. 
\end{theorem}

Put differently, we explicitly construct TQFTs in super vector spaces that can distinguish all diffeomorphism classes of $r$-spin surfaces. 
We also prove (in Section~\ref{subsec:Semisimplicity}) that non-semisimplicity is a necessary condition to produce TQFTs that can distinguish all $r$-spin structures (if~$r>2$). 
This is another illustration that non-semisimple TQFTs are known or expected to resolve finer structure.\footnote{For example, the categorified link invariants of \cite{kr0401268} are computed internal to non-semisimple extended TQFTs of Landau--Ginzburg type (cf.\ \cite[Thm.\,3.7\,\&\,Ex.\,3.13]{CMM}), and semisimple TQFTs are only sensitive to stable diffeomorphism classes as shown in \cite{ReutterSchommerPries2022}.} 

\medskip 

The target $\SVect_{\k}$ is however not good enough to achieve the slightly more ambitious goal of constructing an $r$-spin TQFT~$\Zc$ such that the sequence $(\Zc(\Sigma))_\Sigma$ recovers an arbitrary given allowed sequence of invariants. 
To describe a solution to this problem, and to explain what ``allowed'' means here, we first need to recall the classification of diffeomorphism classes of connected $r$-spin surfaces (see also Theorem~\ref{thm:classification}). 
For $g\geq 0$, let $\Sigma_g$ denote such a closed surface of genus~$g$. 
Then $\Sigma_g$ admits an $r$-spin structure if and only if $2-2g=0$ mod~$r$. 
For $g=0$, the sphere $\Sigma_0$ admits a unique diffeomorphism class of an $r$-spin structure when~$r$ is~1 or~2. 
For $g=1$, there are up to diffeomorphism precisely as many $r$-spin structures as there are divisors~$d$ of~$r$, and we denote the associated invariants by $\beta_d^{\Zc}$. 
For $g>1$ such that $2-2g=0$ mod~$r$, $\Sigma_g$ admits a single $r$-spin diffeomorphism class in case~$r$ is odd, and we write the corresponding invariant as $\alpha_n^{\Zc}$, where $g=nr+1$ for some positive integer~$n$ satisfies $2-2g=0$ mod~$r$. 
If $r$ is even and $g$ satisfies the condition, then up to diffeomorphism there are two $r$-spin structures on $\Sigma_g$, and we write $(\alpha^{+}_n)^{\Zc}$ and $(\alpha^{-}_n)^{\Zc}$ for the invariants that correspond to these two structures, where $g=nr/2+1$ for the appropriate integer~$n$. 
(Section~\ref{subsec:InvariantsFromClosedLambdaRFrobeniusAlgebras} has more details.)

In Section~\ref{sec:symmcat} we construct a family of symmetric monoidal categories, showing that if we do not impose conditions on the target category $\C$, then every sequence can arise from an $r$-spin TQFT. 
Our construction uses that of the symmetric monoidal categories $\C_{\chi}$ which first appeared in \cite{meir21}.
We would like to explain what happens when we restrict the target category. 
Using the terminology of \cite{meir21}, we call an additive $\k$-linear symmetric monoidal category $\C$ \textsl{good} if it is abelian, rigid, and has finite-dimensional Hom spaces. 
The second main theorem in the present paper then is the following:

\begin{theorem}
\label{thm:main}
\begin{enumerate}
\item 
Let $r\in\Z_{\geqslant 1}$ be odd, and let $((b_d)_{d|r},a_1,a_2,\ldots)$ be an infinite sequence in~$\k$. 
There is a good category~$\C$ and an $r$-spin TQFT $\Zc\colon\Bord^r_2\to\C$ such that $\alpha_n^{\Zc} = a_n$ and $\beta_d^\Zc = b_d$ if and only if (for some formal variable~$X$)
\begin{equation}
\label{eq:SpanCondition1}
\sum_{n\in\Z_{\geqslant 1}} a_nX^n
	\;\in\; 
	\operatorname{span}\Big\{\frac{1}{1-\lambda X}\Big\}_{\lambda\in \k} \, . 
\end{equation}
\item 
Let $r\in\Z_{\geqslant 1}$ be even, and let $((b_d)_{d|r},a_1^+,a_1^-,a_2^+,a_2^-,\ldots)$ be an infinite sequence in~$\k$. 
There is a good category~$\C$ and an $r$-spin TQFT $\Zc\colon\Bord^r_2\to\C$ such that $(\alpha_n^+)^{\Zc} = a_n^+$,  $(\alpha_n^-)^{\Zc} = a_n^-$ and $\beta_d^\Zc = b_d$ if and only if 
\begin{equation}
\label{eq:SpanCondition2}
\sum_{n\in\Z_{\geqslant 1}} a_n^{+}X^n\,,\; \sum_{n\in\Z_{\geqslant 1}} a_n^{-}X^n
	\;\in\; 
	\operatorname{span}\Big\{\frac{1}{1-\lambda X}\Big\}_{\lambda\in \k} \,.
\end{equation}
\end{enumerate}
\end{theorem}

Notice that the theorem imposes conditions only on the values of~$\alpha$ (corresponding to surfaces of genus $g>1$), but not on~$\beta$ (corresponding to $g=1$).
The conditions~\eqref{eq:SpanCondition1} and~\eqref{eq:SpanCondition2} follow from a more general criterion on symmetric monoidal categories that are generated by an algebraic structure, constructed in \cite{meir21} and reviewed in Sections~\ref{subsec:BasicNotions}--\ref{subsec:CategoriesCchi}. 
Here we apply that theory to the algebraic structure of closed $\Lambda_r$-Frobenius algebras. 

To prove that any sequence of numbers that satisfies the conditions~\eqref{eq:SpanCondition1} or~\eqref{eq:SpanCondition2} comes from a TQFT, we will show that closed $\Lambda_r$-Frobenius algebras are closed under taking direct sums and a suitable form of tensor product. 
Indeed, Lemma~\ref{lem:plustimes} states that these operations commute nicely with taking invariants in the sense that for two closed $\Lambda_r$-Frobenius algebras~$C$ and~$D$ it holds that $\gamma^{C\oplus D} = \gamma^C+\gamma^D$ and $\gamma^{C\otl D} = \gamma^C\gamma^D$, where $\gamma$ is either of the invariants~$\alpha$ or~$\beta$ (where we use the equivalence of $r$-spin TQFTs and closed $\Lambda_r$-Frobenius algebras). 
This will reduce the proof of Theorem \ref{thm:main} to the construction of enough closed $\Lambda_r$-Frobenius algebras, which in turn is provided in Section~\ref{sec:ExamplesClosedLambdaRFrobeniusAlgebras}. 
Of particular importance are the above-mentioned algebras~$B$ and~$C$ in $\Vect_{\k}$ as well as the Arf algebra~$A$ in $\SVect_{\k}$. 

For $r=1$ the question of what invariants one gets was studied by Khovanov, Ostrik and Kononov in \cite{KOK}. 
Our results here can be considered as an extension of theirs. 
In Section~\ref{subsec:KOK} we explain how the categories constructed in \cite{KOK} are related to the categories we construct here, following \cite{meir21}. 

\medskip 

\noindent
\textbf{Acknowledgements. }
We thank 
	Christoph Schweigert 
for helpful comments on an earlier version of this manuscript. 
N.\,C.\ is supported by the DFG Heisenberg Programme, L.\,S.\ is supported by a DFG Walter Benjamin Fellowship.

\section{$r$-spin surfaces, TQFTs and invariants}\label{sec:rspin}
We summarise the notion of $r$-spin structures, $r$-spin bordisms and the category they form.
Then we consider $r$-spin topological quantum field theories and closed $\Lambda_r$-Frobenius algebras that classify them.
For a more detailed account see e.g.\ \cite{Szegedy-Stern}.

\subsection{$r$-spin surfaces}
\label{subsec:rSpinSurfaces}

For a positive integer $r$, the \textsl{$r$-spin group} is the $r$-fold covering of $\SO_2$:
\begin{align}
  \begin{aligned}
  \Spin_2^r=\SO_2&\lra \SO_2\\
  x&\lmt x^r
  \end{aligned}
  \label{eq:def:r-spin-group}
\end{align}
We call a smooth 2-dimensional manifold a \textsl{surface}. 
An \textsl{$r$-spin structure on a surface $\Sigma$} is a $\Spin_2^r$-principal bundle $P$ over $\Sigma$ together with a bundle map $q\colon P\lra F_{\Sigma}$
to the oriented orthonormal\footnote{Strictly speaking this requires the choice of a metric on $\Sigma$, for details we refer to e.g.\ \cite[Sect.\,2.2]{Szegedy:2018phd}.} frame bundle of $\Sigma$, which is equivariant with respect to \eqref{eq:def:r-spin-group} in the sense that $q(x.p)=x^r.q(p)$.
We call such a triple $(\Sigma,P,q)$ an \textsl{$r$-spin surface}.
We note that the map $q\colon P\lra F_{\Sigma}$ is a $\Z/r$-principal bundle.
Let $(P,q)$ and $(P',q')$ be $r$-spin structures over $\Sigma$. 
An \textsl{isomorphism of $r$-spin structures} $f\colon (P,q)\lra (P',q')$ is a bundle map such that the following diagram commutes:
\begin{equation}
  \begin{tikzcd}[column sep = small]
    P\ar{rr}{f}\ar{dr}[swap]{q}&&P'\ar{dl}{q'}\\
    &F_{\Sigma}&
  \end{tikzcd}
  \label{eq:iso-r-spin-str}
\end{equation}

\begin{example}
  Isomorphism classes of $r$-spin structures on a cylinder $S^1 \times \Rb$ are in bijection with $\Z/r$.
  We write $S^1_x$ for the cylinder with $r$-spin structure corresponding to $x\in\Z/r$.
  \label{ex:r-spin-circle}
\end{example}
Let $(\Sigma,P,q)$ and $(\Sigma',P',q')$ be $r$-spin surfaces.
A \textsl{morphism of $r$-spin surfaces} 
\begin{align}
\Phi\colon (\Sigma,P,q)\lra(\Sigma',P',q')
\end{align}
is a bundle map~$\Phi$ covering a local diffeomorphism $\phi$ such that 
\begin{equation}
  \begin{tikzcd}%[column sep = small]
    P\ar{r}{\Phi}\ar{d}[swap]{q}&P'\ar{d}{q'}\\
    F_{\Sigma}\ar{r}{(\mathrm{d}\phi)^*}\ar{d}{}& F_{\Sigma'}\ar{d}{}\\
    \Sigma\ar{r}[swap]{\phi}&\Sigma'
  \end{tikzcd}
  \label{eq:mor-r-spin-surf}
\end{equation}
commutes, where the map $(\mathrm{d}\phi)^*$ is induced on frames by the differential of $\phi$.
%We call $\Phi$ a \textsl{diffeomorphism of $r$-spin surfaces} if $\phi$ is a diffeomorphism.
If~$\phi$ is a diffeomorphism, then we call~$\Phi$ a \textsl{diffeomorphism of $r$-spin surfaces}; if~$\phi$ is the identity on~$\Sigma$ (hence $\Sigma' = \Sigma$), then we call~$\Phi$ an \textsl{isomorphism}. 
It follows that every isomorphism is a diffeomorphism, and that the number of diffeomorphism classes or $r$-spin structures on a given~$\Sigma$ is bounded from above by the number of isomorphism classes of~$\Sigma$. 

Let $S_0$ and $S_1$ be two closed smooth 1-manifolds and fix $r$-spin structures $(P_i,q_i)$ on $S_i\times\Rb$ for $i\in \{0,1\}$.
Consider open neighbourhoods $U_i\subset S_i\times \Rb$ of $S_i$, set $U_i^+:=U_i\cap S_i\times \Rb_{\geqslant 0}$, $U_i^-:=U_i\cap S_i\times \Rb_{\leqslant 0}$, and let
$(P_i|_{U_i^{\pm}},q_i|_{U_i^{\pm}})$ be the $r$-spin structures restricted to $U_i^{\pm}$.
Let $\Sigma$ be a compact surface (with possibly nonempty boundary) with an $r$-spin structure $(P,q)$.
An \textsl{$r$-spin bordism} 
\begin{equation}
  (\Sigma,P,q)\colon(S_0 \times \Rb,P_0,q_0)\lra(S_1 \times \Rb,P_1,q_1)
  \label{eq:r-spin-bordism}
\end{equation}
consists of the data
\begin{equation}
  \begin{tikzcd}
    \Big(U_0^+, \, P_0\big|_{U_0^+}, \, q_0\big|_{U_0^+}\Big)
  \ar{r}{s_0^\text{in}}
  &(\Sigma,P,q)
  &\Big(U_1^-, \, P_1\big|_{U_1^-}, \, q_1\big|_{U_1^-}\Big)\,,
  \ar{l}[swap]{s_1^\text{out}}
  \end{tikzcd}
  \label{eq:r-spin-bordism-data}
\end{equation}
where the \textsl{boundary parametrisation maps} $s_i^\text{in/out}$ are morphisms of $r$-spin surfaces that identify $S_0\sqcup S_1$ with $\partial\Sigma$. 
We will sometimes write $\Sigma\colon S_0\lra S_1$ for an $r$-spin bordism in case the $r$-spin structures are understood.
A \textsl{diffeomorphism of $r$-spin bordisms} is a diffeomorphism of the compact $r$-spin surfaces that is compatible with the boundary parametrisation maps.

\begin{example}
  The disc $D^2\subset \Rb^2$ has a unique $r$-spin structure, since it is contractible.
  It can be an $r$-spin bordism in two ways:
  \begin{align}
      D^2\colon \varnothing\lra S^1_{+1}
      \quad\text{or}\quad
      D^2\colon S^1_{-1}\lra \varnothing \, ,
    \label{eq:disc-r-spin-bordism}
  \end{align}
  where we write~$S^1_{\pm 1}$ for the codomain or domain circle~$S^1 = \partial D^2$ with the $r$-spin structure induced from~$D^2$. 
\end{example}

Consider two $r$-spin bordisms 
\begin{equation}
  \begin{tikzcd}
  S_0\ar{r}{\Sigma_0}&S_1\ar{r}{\Sigma_1}&S_2
  \end{tikzcd}
  \label{eq:glue-1}
\end{equation}
where the domain of $\Sigma_1$ is equal to the codomain of $\Sigma_0$. 
We define the $r$-spin bordism
\begin{equation}
  \begin{tikzcd}[column sep=3.3em]
  S_0\ar{r}{\Sigma_1\circ\Sigma_0}&S_2
  \end{tikzcd}
  \label{eq:glue-2}
\end{equation}
by gluing the $r$-spin surfaces $\Sigma_0$ and $\Sigma_1$ along the boundary parametrisation maps $s_1^\text{out}$ and $s_1^\text{in}$. 

\begin{definition}
  The \textsl{$r$-spin bordism category} $\Bord_2^r$ consists of the following.
  An object is a closed smooth 1-manifold $S$ with an $r$-spin structure on $S\times \Rb$.
  A morphism $S_0\lra S_1$ is a diffeomorphism class of $r$-spin bordisms.
  Composition is given by gluing $r$-spin bordisms, and identity morphisms are given by cylinders whose boundary parametrisations are embeddings.
  \label{def:bord-r-spin}
\end{definition}

The category $\Bord_2^r$ has a standard monoidal structure given by disjoint union $\sqcup$ with monoidal unit the empty set~$\varnothing$, and a standard symmetric braided structure given by mapping cylinders over swap diffeomorphisms. 
We always view $\Bord_2^r$ as endowed with this symmetric monoidal structure. 

\medskip 

Objects of $\Bord_2^r$ are disjoint unions of $r$-spin circles~$S^1_a$, $a\in\Z/r$. 
We stress that morphisms in $\Bord_2^r$ are diffeomorphism classes, so two distinct \textsl{isomorphism} classes may correspond to the same morphism in $\Bord_2^r$. 
The following theorem explains this in detail for the case of closed $r$-spin surfaces, which by definition represent endomorphisms $\varnothing \lra \varnothing$ in $\Bord_2^r$. 
For more details see e.g.\ \cite[Thm.\,2.2.2]{Szegedy:2018phd} and the original references \cite{Geiges:2012rs,Randal:2014rs}.

\begin{theorem}\label{thm:classification}
  Let $\Sigma_g$ denote a closed connected surface of genus $g$.
  \begin{enumerate}
    \item There exist $r$-spin structures on $\Sigma_g$ if and only if
	\end{enumerate}
      \begin{equation}
	\chi(\Sigma_g)=2 - 2g \equiv 0 \mod{r}\,.
  \label{eq:thm:r-spin-surf-1}
      \end{equation}
  \label{thm:r-spin-surf-1}
Assume that \eqref{eq:thm:r-spin-surf-1} holds. Then
  \begin{enumerate}[resume]
\item 
  the number of isomorphism classes of $r$-spin structures on $\Sigma_g$ is $r^{2g}$, and
\item 
\label{item:DiffeoClasses}
  the number of diffeomorphism classes of $r$-spin structures on $\Sigma_g$ is as follows:
\begin{equation}
\begin{array}{r|c|c}
  g & \text{conditions}& \text{number of diffeomorphism classes}\\
  \hline
0&
r \in \{1,2\} &
 1\\
1&
\text{(none)}&
\#\big( \text{divisors of $r$}\big)\\
\geqslant 2&
r \text{ even}
& 2\\
&r \text{ odd}&
 1
  \end{array}
  \label{eq:thm:r-spin-surf-2}
\end{equation}
  \end{enumerate}
  \label{thm:r-spin-surf}
\end{theorem}

\begin{remark}
	\label{rem:Arf-inv}
  In case $r$ is even and $g\geqslant2$, the two diffeomorphism classes of $r$-spin surfaces with underlying surface $\Sigma_g$ are distinguished by the Arf invariant, cf.\ Section~\ref{sec:Ex_A} below. 
\end{remark}

\subsection{TQFTs and closed \texorpdfstring{$\Lambda_r$}{Lambda\_r}-Frobenius algebras}
\label{subsec:TQFTsClosedLambdaRFrobeniusAlgebras}

Let $\mathcal C \equiv (\mathcal C, \otimes, \one)$ be a symmetric monoidal category, such as the category $\Vect_\k$ of $\k$-vector spaces with monoidal product $\otimes_\k$. 
Another example is $\mathcal C = \SVect_\k$, the category of $\Z/2$-graded vector spaces and even linear maps together with $\otimes_\k$, where the symmetric structure is given by $v\ot w\mapsto (-1)^{|v||w|}w\ot v$. 
By a TQFT we mean a representation of a bordism category on $\mathcal C$. 
More precisely: 

\begin{definition}
	\label{def:rSpinTQFT}
	A \textsl{(closed) $r$-spin TQFT valued in $\mathcal C$} is a symmetric monoidal functor
	$
	\Zc\colon\Bord_2^r \lra \mathcal C
	$. 
\end{definition}

Since we are not concerned with other types of TQFTs (such as ``open'', ``defect'', or ``extended'') in the present paper, we drop the attribute ``closed''. 
We also reiterate that an $r$-spin TQFT may at best distinguish diffeomorphism classes of $r$-spin structures, but it cannot resolve all isomorphism classes (recall the discussion after~\eqref{eq:mor-r-spin-surf}). 

By presenting a bordism category in terms of generators and relations one can equivalently describe TQFTs in terms of algebraic structure in the target category $\mathcal C$. 
In the familiar case of oriented TQFTs, corresponding to $r=1$, this leads to commutative Frobenius algebras. 
For general values of $r$, $r$-spin TQFTs give rise to the following algebraic structure: 

\begin{definition}
	\label{def:ClosedRFrobeniusAlgebra}
	A \textsl{closed $\Lambda_r$-Frobenius algebra}~$C$ 
	in $\mathcal C$ consists of a collection of objects $C_a\in\mathcal C$ for $a\in\Z/r$ as well as morphisms 
	\begin{align}
	\mu_{a,b}=
	%%%%%%%%%%%%%%%%%%%%%%%%%%%%
	\begin{tikzpicture}[very thick,scale=0.53,color=blue!50!black, baseline=0.65cm]
	\draw[-dot-] (2.5,0.75) .. controls +(0,1) and +(0,1) .. (3.5,0.75);
	\draw (3,1.5) -- (3,2.25); 
	\fill[color=black!80] (3,1.5) circle (0pt) node[below] (0up) {{\tiny $a,b$}};
	\end{tikzpicture} 
	%%%%%%%%%%%%%%%%%%%%%%%%%%%% 
	&\colon C_a\otimes C_b\lra C_{a+b-1}\,,
	&\eta_1= 
	%%%%%%%%%%%%%%%%%%%%%%
	\begin{tikzpicture}[very thick,scale=0.4,color=blue!50!black, baseline=0]
	\draw (-0.5,-0.5) node[Odot] (unit) {}; 
	\draw (unit) -- (-0.5,0.5);
	\end{tikzpicture} 
	%%%%%%%%%%%%%%%%%%%%%%
	\label{eq:cLrFa-mor1}
	&\colon\one\lra C_1\,,\\
	\Delta_{a,b}= 
	%%%%%%%%%%%%%%%%%%%%%%%%%%%%
	\begin{tikzpicture}[very thick,scale=0.53,color=blue!50!black, rotate=180, baseline=-0.9cm]
	\draw[-dot-] (2.5,0.75) .. controls +(0,1) and +(0,1) .. (3.5,0.75);
	\draw (3,1.5) -- (3,2.25); 
	\fill[color=black!80] (3,1.5) circle (0pt) node[above] (0up) {{\tiny $a,b$}};
	\end{tikzpicture} 
	%%%%%%%%%%%%%%%%%%%%%%%%%%%% 
	&\colon C_{a+b+1}\lra C_{a}\otimes C_b\,,
	&\varepsilon_{-1}= 
	%%%%%%%%%%%%%%%%%%%%%%
	\begin{tikzpicture}[very thick,scale=0.4,color=blue!50!black, baseline=0cm, rotate=180]
	\draw (-0.5,-0.5) node[Odot] (unit) {}; 
	\draw (unit) -- (-0.5,0.5);
	\end{tikzpicture} 
	%%%%%%%%%%%%%%%%%%%%%%
	&\colon C_{-1}\lra \one
	\label{eq:cLrFa-mor}
	\end{align}
	for $a,b\in\Z/r$.
	The \textsl{Nakayama automorphisms} of~$C$ are 
	\begin{equation}
	N_a:=
		%%%%%%%%%%%%%%%%%%%%%%%%%%%%
		\begin{tikzpicture}[very thick,scale=0.53,color=blue!50!black, baseline=-0.4cm]
		\draw[-dot-] (0,0) .. controls +(0,1) and +(0,1) .. (1,0);
		\draw (0.5,1.2) node[Odot] (unit) {}; 
		\draw (unit) -- (0.5,0.7);
		\draw (0,0) .. controls +(0,-0.5) and +(0,0.5) .. (-1,-1);
		\draw[-dot-] (0,-1) .. controls +(0,-1) and +(0,-1) .. (1,-1);
		\draw (0.5,-2.2) node[Odot] (unit2) {}; 
		\draw (unit2) -- (0.5,-1.7);
		\draw (0,-1) .. controls +(0,0.5) and +(0,-0.5) .. (-1,0);
		\draw (1,0) -- (1,-1);
		\draw (-1,0) -- (-1,2);
		\draw (-1,-1) -- (-1,-3);
		\fill[color=black!80] (0.55,-1.75) circle (0pt) node[right] (0up) {{\tiny $a,- a$}};
		\fill[color=black!80] (0.55,0.75) circle (0pt) node[right] (0up) {{\tiny $a,- a$}};
		\end{tikzpicture} 
		%%%%%%%%%%%%%%%%%%%%%%%%%%%% 
	\colon C_a\lra C_a \, . 
	\label{eq:Nakayama-Ca}
	\end{equation}
	These data must satisfy the following conditions:
	\begin{align}
	\textrm{(co)associativity:} &&  
	%%%%%%%%%%%%%%%%%%%%%%%%%%%% associativity
	\begin{tikzpicture}[very thick,scale=0.53,color=blue!50!black, baseline=0.4cm]
	\draw[-dot-] (3,0) .. controls +(0,1) and +(0,1) .. (2,0);
	\draw[-dot-] (2.5,0.75) .. controls +(0,1) and +(0,1) .. (3.5,0.75);
	\draw (3.5,0.75) -- (3.5,0); 
	\draw (3,1.5) -- (3,2.25); 
	\fill[color=black!80] (2.5,0.75) circle (0pt) node[below] (0up) {{\tiny $a,b$}};
	\fill[color=black!80] (3,1.5) circle (0pt) node[left] (0up) {{\tiny $a\!+\!b\!-\!1,c$}};
	\end{tikzpicture} 
	%%%%%%%%%%%%%%%%%%%%%%%%%%%% 
	=
	%%%%%%%%%%%%%%%%%%%%%%%%%%%%
	\begin{tikzpicture}[very thick,scale=0.53,color=blue!50!black, baseline=0.4cm]
	\draw[-dot-] (3,0) .. controls +(0,1) and +(0,1) .. (2,0);
	\draw[-dot-] (2.5,0.75) .. controls +(0,1) and +(0,1) .. (1.5,0.75);
	\draw (1.5,0.75) -- (1.5,0); 
	\draw (2,1.5) -- (2,2.25); 
	\fill[color=black!80] (2.5,0.75) circle (0pt) node[below] (0up) {{\tiny $b,c$}};
	\fill[color=black!80] (2,1.5) circle (0pt) node[right] (0up) {{\tiny $a,b\!+\!c\!-\!1$}};
	\end{tikzpicture} 
	%%%%%%%%%%%%%%%%%%%%%%%%%%%% 
	\, , 
	\quad
	%%%%%%%%%%%%%%%%%%%%%% coassociativity
	\begin{tikzpicture}[very thick,scale=0.53,color=blue!50!black, baseline=-0.8cm, rotate=180]
	\draw[-dot-] (3,0) .. controls +(0,1) and +(0,1) .. (2,0);
	\draw[-dot-] (2.5,0.75) .. controls +(0,1) and +(0,1) .. (1.5,0.75);
	\draw (1.5,0.75) -- (1.5,0); 
	\draw (2,1.5) -- (2,2.25); 
	\fill[color=black!80] (2.5,0.75) circle (0pt) node[above] (0up) {{\tiny $a,b$}};
	\fill[color=black!80] (2,1.5) circle (0pt) node[left] (0up) {{\tiny $a\!+\!b\!+\!1,c$}};
	\end{tikzpicture} 
	%%%%%%%%%%%%%%%%%%%%%% 
	=
	%%%%%%%%%%%%%%%%%%%%%%
	\begin{tikzpicture}[very thick,scale=0.53,color=blue!50!black, baseline=-0.8cm, rotate=180]
	\draw[-dot-] (3,0) .. controls +(0,1) and +(0,1) .. (2,0);
	\draw[-dot-] (2.5,0.75) .. controls +(0,1) and +(0,1) .. (3.5,0.75);
	\draw (3.5,0.75) -- (3.5,0); 
	\draw (3,1.5) -- (3,2.25); 
	\fill[color=black!80] (2.5,0.75) circle (0pt) node[above] (0up) {{\tiny $b,c$}};
	\fill[color=black!80] (3,1.5) circle (0pt) node[right] (0up) {{\tiny $a,b\!+\!c\!+\!1$}};
	\end{tikzpicture} 
	%%%%%%%%%%%%%%%%%%%%%% 
	&
	\label{eq:clLrFa-co-associativity}
	\\
	\textrm{(co)unitality:} &&  
	%%%%%%%%%%%%%%%%%%%%%% unitality
	\begin{tikzpicture}[very thick,scale=0.4,color=blue!50!black, baseline]
	\draw (-0.5,-0.5) node[Odot] (unit) {}; 
	\fill (0,0.6) circle (5.0pt) node (meet) {};
	\fill[color=black!80] (meet) circle (0pt) node[left] (0up) {{\tiny $1,a$}};
	\draw (unit) .. controls +(0,0.5) and +(-0.5,-0.5) .. (0,0.6);
	\draw (0,-1.5) -- (0,1.5); 
	\end{tikzpicture} 
	%%%%%%%%%%%%%%%%%%%%%% 
	=
	%%%%%%%%%%%%%%%%%%%%%%%%%%%%
	\begin{tikzpicture}[very thick,scale=0.4,color=blue!50!black, baseline]
	\draw (0,-1.5) -- (0,1.5); 
	\end{tikzpicture} 
	%%%%%%%%%%%%%%%%%%%%%%%%%%%% 
	=
	%%%%%%%%%%%%%%%%%%%%%%
	\begin{tikzpicture}[very thick,scale=0.4,color=blue!50!black, baseline]
	\draw (0.5,-0.5) node[Odot] (unit) {}; 
	\fill (0,0.6) circle (5.0pt) node (meet) {};
	\draw (unit) .. controls +(0,0.5) and +(0.5,-0.5) .. (0,0.6);
	\draw (0,-1.5) -- (0,1.5); 
	\fill[color=black!80] (meet) circle (0pt) node[right] (0up) {{\tiny $a,1$}};
	\end{tikzpicture} 
	%%%%%%%%%%%%%%%%%%%%%% 
	\,,
	\quad
	\quad
	%%%%%%%%%%%%%%%%%%%%%% counitality
	\begin{tikzpicture}[very thick,scale=0.4,color=blue!50!black, baseline=0, rotate=180]
	\draw (0.5,-0.5) node[Odot] (unit) {}; 
	\fill (0,0.6) circle (5.0pt) node (meet) {};
	\draw (unit) .. controls +(0,0.5) and +(0.5,-0.5) .. (0,0.6);
	\draw (0,-1.5) -- (0,1.5); 
	\fill[color=black!80] (meet) circle (0pt) node[left] (0up) {{\tiny $-1,a$}};
	\end{tikzpicture} 
	%%%%%%%%%%%%%%%%%%%%%% 
	=
	%%%%%%%%%%%%%%%%%%%%%%
	\begin{tikzpicture}[very thick,scale=0.4,color=blue!50!black, baseline=0, rotate=180]
	\draw (0,-1.5) -- (0,1.5); 
	\end{tikzpicture} 
	%%%%%%%%%%%%%%%%%%%%%% 
	=
	%%%%%%%%%%%%%%%%%%%%%%
	\begin{tikzpicture}[very thick,scale=0.4,color=blue!50!black, baseline=0cm, rotate=180]
	\draw (-0.5,-0.5) node[Odot] (unit) {}; 
	\fill (0,0.6) circle (5.0pt) node (meet) {};
	\draw (unit) .. controls +(0,0.5) and +(-0.5,-0.5) .. (0,0.6);
	\draw (0,-1.5) -- (0,1.5); 
	\fill[color=black!80] (meet) circle (0pt) node[right] (0up) {{\tiny $a,-1$}};
	\end{tikzpicture} 
	%%%%%%%%%%%%%%%%%%%%%% 
	&
	\label{eq:clLrFa-co-unitality}
	\\
	\textrm{Frobenius relation:} &&  
	%%%%%%%%%%%%%%%%%%%%%% Frobenius
	\begin{tikzpicture}[very thick,scale=0.4,color=blue!50!black, baseline=0cm]
	\draw[-dot-] (0,0) .. controls +(0,-1) and +(0,-1) .. (-1,0);
	\draw[-dot-] (1,0) .. controls +(0,1) and +(0,1) .. (0,0);
	\draw (-1,0) -- (-1,1.5); 
	\draw (1,0) -- (1,-1.5); 
	\draw (0.5,0.8) -- (0.5,1.5); 
	\draw (-0.5,-0.8) -- (-0.5,-1.5); 
	\fill[color=black!80] (-0.5,-0.8) circle (0pt) node[left] (0up) {{\tiny $c,a\!-\!c\!-\!1$}};
	\fill[color=black!80] (0.5,0.8) circle (0pt) node[right] (0up) {{\tiny $d\!-\!b\!+\!1,b$}};
	\end{tikzpicture}
	%%%%%%%%%%%%%%%%%%%%%%
	=
	%%%%%%%%%%%%%%%%%%%%%%
	\begin{tikzpicture}[very thick,scale=0.4,color=blue!50!black, baseline=0cm]
	\draw[-dot-] (0,0) .. controls +(0,1) and +(0,1) .. (-1,0);
	\draw[-dot-] (1,0) .. controls +(0,-1) and +(0,-1) .. (0,0);
	\draw (-1,0) -- (-1,-1.5); 
	\draw (1,0) -- (1,1.5); 
	\draw (0.5,-0.8) -- (0.5,-1.5); 
	\draw (-0.5,0.8) -- (-0.5,1.5); 
	\fill[color=black!80] (0.5,-0.8) circle (0pt) node[right] (0up) {{\tiny $b\!-\!d\!-\!1,d$}};
	\fill[color=black!80] (-0.5,0.8) circle (0pt) node[left] (0up) {{\tiny $a,c\!-\!a\!+\!1$}};
	\end{tikzpicture}
	%%%%%%%%%%%%%%%%%%%%%%
	&
	\label{eq:clLrFa-Frobenius}
	\\
	\textrm{commutativity:} &&  
		%%%%%%%%%%%%%%%%%%%%%%%%%%%%
		\begin{tikzpicture}[very thick,scale=0.53,color=blue!50!black, baseline=0cm]
		\draw[-dot-] (-0.5,0) .. controls +(0,1) and +(0,1) .. (0.5,0);
		\draw (0,0.7) -- (0,1.5);
		\draw (-0.5,0) -- (-0.5,-0.75);
		\draw (0.5,0) -- (0.5,-0.75);
		\fill (-0.5,-0.25) circle (3.5pt) node[left] (mult1) {{\tiny $N_b^{1-a}$}};
		\fill[color=black!80] (0,0.7) circle (0pt) node[below] (0up) {{\tiny $b,\!a$}};
		\end{tikzpicture} 
		%%%%%%%%%%%%%%%%%%%%%%%%%%%% 
	= 
		%%%%%%%%%%%%%%%%%%%%%%%%%%%%
		\begin{tikzpicture}[very thick,scale=0.53,color=blue!50!black, baseline=0cm]
		\draw[-dot-] (-0.5,0) .. controls +(0,1) and +(0,1) .. (0.5,0);
		\draw (0,0.7) -- (0,1.5);
		\draw (-0.5,0) .. controls +(0,-0.25) and +(0,0.25) .. (0.5,-0.75);
		\draw (0.5,0) .. controls +(0,-0.25) and +(0,0.25) .. (-0.5,-0.75);
		%
		%	\fill (-0.85,0.2) circle (2.5pt) node (mult1) {};
		\fill[color=black!80] (0,0.7) circle (0pt) node[below] (0up) {{\tiny $a,\!b$}};
		\end{tikzpicture} 
		%%%%%%%%%%%%%%%%%%%%%%%%%%%% 
	= 
		%%%%%%%%%%%%%%%%%%%%%%%%%%%%
		\begin{tikzpicture}[very thick,scale=0.53,color=blue!50!black, baseline=0cm]
		\draw[-dot-] (-0.5,0) .. controls +(0,1) and +(0,1) .. (0.5,0);
		\draw (0,0.7) -- (0,1.5);
		\draw (-0.5,0) -- (-0.5,-0.75);
		\draw (0.5,0) -- (0.5,-0.75);
		\fill (0.5,-0.25) circle (3.5pt) node[right] (mult1) {{\tiny $N_a^{b-1}$}};
		\fill[color=black!80] (0,0.7) circle (0pt) node[below] (0up) {{\tiny $b,\!a$}};
		\end{tikzpicture} 
		%%%%%%%%%%%%%%%%%%%%%%%%%%%% 
	&
	\label{eq:clLrFa-commutativity}
	\\
	\textrm{twist relations:} &&  
	N_a^a = 1_{C_a}\,,
	\quad
	\quad
	%\\
	%&&  
		%%%%%%%%%%%%%%%%%%%%%%%%%%%%
		\begin{tikzpicture}[very thick,scale=0.53,color=blue!50!black, baseline=0cm]
		\draw[-dot-] (-0.5,0) .. controls +(0,1) and +(0,1) .. (0.5,0);
		\draw (0,0.7) -- (0,1.5);
		\draw[-dot-] (-0.5,0) .. controls +(0,-1) and +(0,-1) .. (0.5,0);
		\draw (0,-1.2) node[Odot] (unit) {}; 
		\draw (unit) -- (0,-0.8);
		\fill (-0.5,0) circle (3.5pt) node[left] (mult1) {{\tiny $N_a^{b}$}};
		\fill[color=black!80] (0,0.8) circle (0pt) node[left] (0up) {{\tiny $a,\!-a$}};
		\fill[color=black!80] (0,-0.8) circle (0pt) node[left] (0up) {{\tiny $a,\!-a$}};
		\end{tikzpicture} 
		%%%%%%%%%%%%%%%%%%%%%%%%%%%% 
	=
		%%%%%%%%%%%%%%%%%%%%%%%%%%%%
		\begin{tikzpicture}[very thick,scale=0.53,color=blue!50!black, baseline=0cm]
		\draw[-dot-] (-0.5,0) .. controls +(0,1) and +(0,1) .. (0.5,0);
		\draw (0,0.7) -- (0,1.5);
		\draw[-dot-] (-0.5,0) .. controls +(0,-1) and +(0,-1) .. (0.5,0);
		\draw (0,-1.2) node[Odot] (unit) {}; 
		\draw (unit) -- (0,-0.8);
		\fill (-0.5,0) circle (3.5pt) node[left] (mult1) {{\tiny $N_{a+b-1}^{b}$}};
		\fill[color=black!80] (0,0.8) circle (0pt) node[left] (0up) {{\tiny $a+b-1,\!-a-b+1$}};
		\fill[color=black!80] (0,-0.8) circle (0pt) node[left] (0up) {{\tiny $a+b-1,\!-a-b+1$}};
		\end{tikzpicture} 
		%%%%%%%%%%%%%%%%%%%%%%%%%%%% 
	&
	\label{eq:clLrFa-twist}
	\end{align}
	and the deck transformation relations $N_a^r = 1_{C_a}$. 
\end{definition}

Note that a closed $\Lambda_r$-Frobenius algebra satisfies the following non-degeneracy condition:
For every $a\in\Z/r$ there are morphisms $\delta_a:=\Delta_{a,-a}\circ\eta_{1}\colon \one\lra C_{a}\otimes C_{-a}$ such that
\begin{align}
  (1_{C_a}\otimes \varepsilon_{-1}\circ\mu_{-a,a})\circ (\delta_{a}\otimes 1_{C_a})
  =1_{C_a}=
  (\varepsilon_{-1}\circ\mu_{a,-a}\otimes 1_{C_a})\circ (1_{C_a}\otimes \delta_{-a})\,.
  \label{eq:clLrFa-nondeg}
\end{align}
This condition determines $\delta_a$ uniquely and provides an alternative definition of closed $\Lambda_r$-Frobenius algebras.

Assume that finite direct sums exist in~$\mathcal C$, and consider an object $D=\bigoplus_{a\in\Z/r}D_a$ together with morphisms $\mu_{a,b}$, $\eta_{1}$ and $\varepsilon_{-1}$ as in \eqref{eq:cLrFa-mor1} and \eqref{eq:cLrFa-mor} satisfying associativity \eqref{eq:clLrFa-co-associativity}, unitality \eqref{eq:clLrFa-co-unitality}, and the non-degeneracy condition \eqref{eq:clLrFa-nondeg} for some morphisms $\delta_{a}$.
Furthermore assume that the above morphisms satisfy \eqref{eq:clLrFa-commutativity}, \eqref{eq:clLrFa-twist} and the deck transformation relations with $\Delta_{a,-a}\circ\eta_1$ replaced by~$\delta_a$.
The proof of the following lemma is completely analogous to the case of ordinary Frobenius algebras, see e.g.\ \cite{Kockbook}.

\begin{lemma}
  The object~$D$ as above together with the morphisms $\mu_{a,b}$, $\eta_1$, $\varepsilon_{-1}$ and
  \begin{align}
    \Delta_{a,b}:=(1_{D_a}\otimes \mu_{-a,a+b+1})\circ(\delta_a \otimes 1_{D_a})
    \label{eq:lem:alt-def-clLrFa}
  \end{align}
  form a closed $\Lambda_r$-Frobenius algebra.
  \label{lem:alt-def-clLrFa}
\end{lemma}

In $\mathcal C = \Bord_2^r$, there is a particular closed $\Lambda_r$-Frobenius algebra, where $C_a = S_a^1$, $\mu_{a,b}$ and~$\Delta_{a,b}$ are given by (upside-down) pairs-of-pants with $r$-structure, $\eta_1$ and~$\varepsilon_{-1}$ are given by discs as in~\eqref{eq:disc-r-spin-bordism}, and~$N_a$ corresponds to a deck transformation on~$S_a^1$. 
This generalises the fact that circles, pairs-of-pants and discs form a commutative Frobenius algebra in the category of 2-dimensional oriented bordisms, to which the preceding statement reduces for $r=1$. 

More generally, closed $\Lambda_r$-Frobenius algebras in an arbitrary category~$\mathcal C$ form the objects of a category $\Lambda_r\!\operatorname{Frob}(\mathcal C)$, whose morphisms $\varphi\colon C\lra D$ by definition are collections of morphisms $\varphi_a\colon C_a\lra D_a$ in~$\mathcal C$ which preserve the structure morphisms.
Analogously to the case $r=1$, one finds that $\Lambda_r\!\operatorname{Frob}(\mathcal C)$ is a groupoid. 
Moreover, as shown in \cite[Cor.\,5.2.2]{Szegedy-Stern}, pairs-of-pants and discs in fact generate $\Bord_2^r$, hence closed $\Lambda_r$-Frobenius algebras ``classify'' $r$-spin TQFTs:\footnote{The groupoid of $r$-spin TQFTs valued in~$\mathcal C$ inherits a symmetric monoidal structure from~$\mathcal C$. Then the equivalence of \cite{Szegedy-Stern} induces a (compatible) symmetric monoidal structure on $\Lambda_r\!\operatorname{Frob}(\mathcal C)$ implied in Theorem~\ref{thm:closed-r-spin-TQFT-classification}. Below in Section~\ref{sec:Ex_op} we describe this tensor product of closed $\Lambda_r$-Frobenius algebras explicitly.}

\begin{theorem}
	\label{thm:closed-r-spin-TQFT-classification}
	There is an equivalence of groupoids between $r$-spin TQFTs valued in~$\mathcal C$ and $\Lambda_r\!\operatorname{Frob}(\mathcal C)$.
\end{theorem}

To an $r$-spin TQFT $\mathcal Z$, this equivalence assigns the closed $\Lambda_r$-Frobenius algebra $C^{\mathcal Z}$ with $C^{\mathcal Z}_a := \mathcal Z(S_a^1)$ and whose structure maps \eqref{eq:cLrFa-mor1}--\eqref{eq:cLrFa-mor} are given by the images of the generators of $\Bord_2^r$. 
These generators by construction form a closed $\Lambda_r$-Frobenius algebra in $\Bord_2^r$. 
Conversely, from a given object $C \in \Lambda_r\!\operatorname{Frob}(\mathcal C)$ one obtains a TQFT which assigns the data of $C$ to the generators of $\Bord_2^r$.\footnote{As explained in \cite{Szegedy:2018phd}, the case $r=0$ corresponds to \textsl{framed} TQFTs. \textsl{Unoriented} 2-dimensional TQFTs were similarly classified in \cite{TuraevTurnerUnoriented}, and \cite{Czenky} gives a detailed analysis of this case in analogy to the work \cite{KOK} on the oriented case.}
We discuss examples of closed $\Lambda_r$-Frobenius algebras in Section~\ref{sec:ExamplesClosedLambdaRFrobeniusAlgebras} below. 

\medskip 

There is a canonical construction of an $r$-spin TQFT from an $s$-spin TQFT if~$s$ divides~$r$. 
Indeed, in this case we have a surjective group homomorphism $p_{r,s}\colon \Spin_2^r \lra \Spin_2^s$ given by $x\lmt x^{r/s}$. 
Pushing forward gives us a functor
\begin{equation}
(p_{r,s})_* \colon \Bord_2^r \lra \Bord_2^s \quad (\textrm{for } s|r) \, , 
\end{equation}
hence pre-composing with $(p_{r,s})_*$ assigns an $r$-spin TQFT to any $s$-spin TQFT. 
On the algebraic side, this translates into a functor 
\begin{equation}
\label{eq:PullbackRS}
P_{r,s}^* \colon \Lambda_s\!\operatorname{Frob}(\mathcal C) \lra \Lambda_r\!\operatorname{Frob}(\mathcal C)  \quad (\textrm{for } s|r) 
\end{equation}
with 
\begin{equation}
\big( P_{r,s}^*(C) \big)_a = C_{a\!\!\mod\!\!s} \, , 
\end{equation}
and the structure maps of $P_{r,s}^*(C)$ are induced by those of $C\in\Lambda_s\!\operatorname{Frob}(\mathcal C)$ in a straightforward way.

\subsection{Invariants from closed \texorpdfstring{$\Lambda_r$}{Lambda\_r}-Frobenius algebras}
\label{subsec:InvariantsFromClosedLambdaRFrobeniusAlgebras}

Every $r$-spin TQFT $\mathcal Z$ valued in $\mathcal C$ in particular provides invariants $\mathcal Z(\Sigma)$ of $r$-spin surfaces $\Sigma$. 
To compute them, we present $\Sigma$ in terms of the generator bordisms (i.e.\ pairs-of-pants, discs, and mapping cylinders over deck transformations), and consider the appropriate composition of the data of the closed $\Lambda_r$-Frobenius algebra $C^{\mathcal Z}$ (i.e.\ $\mu^{\mathcal Z}_{a,b}, \Delta_{a,b}^{\mathcal Z}, \eta_1^{\mathcal Z},  \varepsilon_{-1}^{\mathcal Z}$, and $N_a^{\mathcal Z}$, respectively). 
In this section we fix a TQFT~$\mathcal Z$ as above and compute the invariants it provides for diffeomorphism classes of closed $r$-spin surfaces, i.e.\ for morphisms $\varnothing\lra\varnothing$ in $\Bord_2^r$. 

For example, a sphere $\Sigma = S^2$ admits $r$-spin structures only if $r\in\{1,2\}$, of which there is then only one up diffeomorphism according to Theorem \ref{thm:r-spin-surf}(\ref{item:DiffeoClasses}). 
Since $S^2$ is obtained by glueing two discs together, the associated invariant is 
\begin{equation}
\label{eq:SphereInvariant}
\mathcal Z\big( S^2 \big) 
= 
\varepsilon_{-1}^{\mathcal Z} \circ \eta_1^{\mathcal Z} 
= 
%%%%%%%%%%%%%%%%%%%%%%
\begin{tikzpicture}[very thick,scale=0.4,color=blue!50!black, baseline=-0.1cm]
\draw (-0.5,-0.5) node[Odot] (unit) {}; 
\draw (-0.5,0.5) node[Odot] (counit) {}; 
\draw (unit) -- (counit);
\end{tikzpicture} 
%%%%%%%%%%%%%%%%%%%%%% 
\quad 
\textrm{for } r\in\{1,2\}
\, . 
\end{equation}

The most interesting case is when $\Sigma$ is a torus. 
As described in \cite[Sect.\,4.1]{Szegedy-Stern}, every $r$-spin torus can be obtained by glueing two bent cylinders over an $r$-spin circle $S_a^1$ after twisting with the $(1-b)$-th power of a deck transformation cylinder, for some $a,b\in\Z/r$. 
Such a torus gives rise to an endomorphism of $\varnothing \in \Bord_2^r$ which we denote $T(a,b)$. 
In line with the $g=1$ case of Theorem \ref{thm:r-spin-surf}(\ref{item:DiffeoClasses}) one finds 
\begin{equation}
T(a,b) = T\big(\gcd(a,b,r), 0 \big) \, . 
\end{equation}
Hence it is enough to compute invariants of $T(d) := T(d,0)$ for $d := \gcd(a,b,r)$. 
For a TQFT $\mathcal Z$ as above these invariants are given by 
\begin{equation}\label{eq:beta}
\beta^{\mathcal Z}_d
:= 
\mathcal Z \big(T(d)\big) 
= 
\varepsilon_{-1}^{\mathcal Z} \circ \mu_{-a,a}^{\mathcal Z} \circ \big( (N_{-a}^{\mathcal Z})^{1-b} \otimes 1_{C^{\mathcal Z}_a} \big) \circ \Delta_{-a,a}^{\mathcal Z} \circ \eta_1^{\mathcal Z}
= 
%%%%%%%%%%%%%%%%%%%%%%%%%%%%
\begin{tikzpicture}[very thick,scale=0.53,color=blue!50!black, baseline=-0.1cm]
\draw[-dot-] (-0.5,0) .. controls +(0,1) and +(0,1) .. (0.5,0);
\draw (0,1.2) node[Odot] (unit2) {}; 
\draw (unit2) -- (0,0.7);
\draw[-dot-] (-0.5,0) .. controls +(0,-1) and +(0,-1) .. (0.5,0);
\draw (0,-1.2) node[Odot] (unit) {}; 
\draw (unit) -- (0,-0.8);
\fill (-0.5,0) circle (3.5pt) node[left] (mult1) {{\tiny $(N_{-a}^{\mathcal Z})^{1-b}$}};
\fill[color=black!80] (0,0.8) circle (0pt) node[left] (0up) {{\tiny $-a,\!a$}};
\fill[color=black!80] (0,-0.8) circle (0pt) node[left] (0up) {{\tiny $-a,\!a$}};
\end{tikzpicture} 
%%%%%%%%%%%%%%%%%%%%%%%%%%%% 
\end{equation}
which does not depend on the individual values of $a,b$. 
Indeed, as shown in \cite[Prop.\,4.1.4]{Szegedy-Stern}, this is equal to the quantum dimension of $C^{\mathcal Z}_d$ in $\mathcal C$: 
\begin{equation}
\label{eq:betaInvariants}
\beta^{\mathcal Z}_d 
= 
\dim \!\big( C^{\mathcal Z}_d \big) 
= 
\operatorname{ev}_{C^{\mathcal Z}_d} \circ \, b_{C^{\mathcal Z}_{-d}, C^{\mathcal Z}_d} \circ \operatorname{coev}_{C^{\mathcal Z}_d} 
\end{equation}
where $b$ denotes the braiding. 

\medskip 

According to Theorem~\ref{thm:r-spin-surf-1}, a closed surface of genus $g\geqslant 2$ admits an $r$-spin structure if and only if $2-2g \equiv 0 \mod r$. 
The number of diffeomorphism classes of $r$-spin structures then only depends on the parity of~$r$. 

We first discuss the case where~$r$ is odd. 
Then for fixed $g\in\Z_{\geqslant 2}$ a closed surface~$\Sigma_g$ of genus~$g$ admits precisely one diffeomorphism class of $r$-spin structures. 
To compute the invariant associated to~$\Sigma_g$ by~$\mathcal Z$, we first note that the ``handle map'' 
\begin{equation}
\label{eq:HandleMap}
h_{x,a,b}^{\mathcal Z} := \mu_{a,b}^{\mathcal Z} \circ \Delta_{a,b}^{\mathcal Z} \colon C_x^{\mathcal Z} \lra C_{x-2}^{\mathcal Z} 
	\quad 
	\textrm{for } x = a+b+1
\end{equation}
may depend on the individual values of $a,b$, and not only on $a+b+1\in\Z/r$. Choose numbers $a_i,b_i\in \Z/r$, $i=0,\ldots, r-1$ such that $a_i+b_i+1 = x-2i$.  
Since~$r$ is odd, the composition 
\begin{equation}
\label{eq:Hx}
H_x^{\mathcal Z} 
	:= 
	h^{\mathcal Z}_{x+2,a_{r-1},b_{r-1}} \circ h^{\mathcal Z}_{x+4,a_{r-2},b_{r-2}} \circ \dots \circ h^{\mathcal Z}_{x-2,a_1,b_1} \circ h^{\mathcal Z}_{x,a_0,b_0} \colon C_x^{\mathcal Z} \lra C_x^{\mathcal Z}
\end{equation}
has precisely~$r$ factors. 
The map $H_x^{\mathcal Z}$ is the image under~$\mathcal Z$ of a connected bordism $S_x^1 \lra S_x^1$ of genus~$r$, and as shown in \cite[Thm.\,2.2.2\,(3)]{Szegedy:2018phd} it only depends on $x\in\Z/r$, not on the individual numbers $a_i$ and $b_i$. 
If compositions of maps of type $h_{x,a,b}$ do not depend on $a,b$, we suppress these indices below. 
Hence by composing $n\in\Z_{\geqslant 1}$ copies of $H_x^{\mathcal Z}$ and then taking the trace in~$\mathcal C$ (where taking the trace corresponds to adding another handle by identifying the in- and outgoing boundaries~$S_x^1$), we find that 
\begin{equation}
\label{eq:alphaOdd}
\alpha_n^{\mathcal Z} 
	:= \mathcal Z (\Sigma_{nr+1}) 
	= \textrm{tr} \Big( \big[ h^{\mathcal Z}_{x+2} \circ h^{\mathcal Z}_{x+4} \circ \dots \circ h^{\mathcal Z}_{x-2} \circ h^{\mathcal Z}_{x} \big]^n \Big) 
\end{equation}
is the invariant associated by~$\mathcal Z$ to~$\Sigma_{nr+1}$. 
Note that $g:=nr+1$ indeed satisfies $2-2g \equiv 0 \mod r$. 

Next we assume that~$r$ is even and $g:=nr/2+1 \in \Z_{\geqslant 2}$ for some $n\in\Z_{\geqslant 1}$. 
Then according to Theorem~\ref{thm:r-spin-surf-1}, $\Sigma_g$ admits precisely two diffeomorphism classes of $r$-spin structures. 
These two distinct morphisms $\varnothing\lra\varnothing$ in $\Bord_2^r$ are mapped to 
\begin{align}
\label{eq:alphaEven}
(\alpha_n^+)^{\mathcal Z} 
:= \mathcal Z (\Sigma_{nr/2 + 1}^+)
	= \textrm{tr} \Big( \big[ &(h^+_{x+2})^{\mathcal Z} \circ (h^{+}_{x+4})^{\mathcal Z} \circ \dots \circ (h^{+}_{x-2})^{\mathcal Z} \circ (h^{+}_{x})^{\mathcal Z} \big]^n \Big) 
	\\ 
(\alpha_n^-)^{\mathcal Z}
	:= \mathcal Z (\Sigma_{nr/2 + 1}^-)
	= \textrm{tr} \Big( \big[ &(h^+_{x+2})^{\mathcal Z} \circ (h^{+}_{x+4})^{\mathcal Z} \circ \dots \circ (h^{+}_{x-2})^{\mathcal Z} \circ (h^{+}_{x})^{\mathcal Z} \big]^{n-1} \nonumber
	\\ 
       &
       \qquad
       \circ(h^+_{x+2})^{\mathcal Z} \circ (h^{+}_{x+4})^{\mathcal Z} \circ \dots \circ (h^{+}_{x-2})^{\mathcal Z} \circ (h^{-}_{x})^{\mathcal Z} \Big) 
       \label{eq:alphaEvenMinus}
\end{align}
where 
\begin{align}
    (h_x^+)^{\mathcal Z} &:= \mu_{0,x-1}^{\mathcal Z} \circ \Delta_{0,x-1}^{\mathcal Z} \,,     \\
   (h_x^-)^{\mathcal Z} &:= \mu_{0,x-1}^{\mathcal Z} \circ \big( N_0 \otimes 1_{C_{x-1}^{\mathcal Z}} \big) \circ \Delta_{0,x-1}^{\mathcal Z}\,,     
\end{align}
are both maps $C_x^{\mathcal Z} \lra C_{x-2}^{\mathcal Z}$. 
We can define the maps $(h_x^{\pm})^{\mathcal Z}$ using general $a,b$ and not necessarily $a=0$ and $b=x-1$, as long as $a$ is even (similarly to what we have done in the odd case). Again, $(\alpha_n^\pm)^{\mathcal Z}$ depends only on~$n$, and not on the ``intermediate'' variables. 
By evaluating both morphisms on the algebra $A$ from Section~\ref{sec:Ex_A} we see that $(\alpha_n^+)^{\mathcal Z}$ and $(\alpha_n^-)^{\mathcal Z}$ correspond indeed to the two $r$-spin surfaces $\Sigma_{nr/2+1}^{\pm}$.
By a similar argument and \cite[Thm.\,2.2.2]{Szegedy:2018phd} we have the following:

\begin{lemma}\label{lem:identities}
Let $r$ be even, and let $x\in \Z/r$. It holds that $h_{x-2}^{-}h_x^{+} = h_{x-2}^+h_x^-$ and
$h_{x-2}^{+}h_x^{+} = h_{x-2}^-h_x^-$ 
\end{lemma}

In summary, an $r$-spin TQFT~$\mathcal Z$ assigns the invariants~\eqref{eq:SphereInvariant}, \eqref{eq:betaInvariants}, \eqref{eq:alphaOdd} and~\eqref{eq:alphaEven} to closed $r$-spin surfaces, computed in terms of the associated closed $\Lambda_r$-Frobenius algebra~$C^{\mathcal Z}$. 
Conversely, by dropping the superindices ``$\mathcal Z$'' in these formulas, we obtained invariants from any $C\in \Lambda_r\!\operatorname{Frob}(\mathcal C)$. 

\medskip 

Recall that for $s|r$, we have the pullback functor $P_{r,s}^*$ in~\eqref{eq:PullbackRS}. 
It is straightforward to compute the invariants of a TQFT associated to an image under $P_{r,s}^*$: 

\begin{lemma}
  Let $s|r$, $C \in \Lambda_s\!\operatorname{Frob}(\mathcal C)$, and set $D:=P_{r,s}^*(C)$. 
  Then
  \begin{align}
%    \begin{aligned}
      \beta_d^D&=
      \beta_{d\!\!\mod\!\!s}^C 
      &&
      \textrm{for } d|s,
      \\
      \alpha_n^{D}&=
      \alpha_{nr/s}^{C} 
      &&
      \textrm{if both $r$ and $s$ are odd},
      \\
      \big(\alpha_n^{\pm}\big)^D&=
      \alpha_{nr/(2s)}^{C} 
      &&
      \textrm{if $r$ is even, $s$ is odd},
      \\
      \big(\alpha_n^{\pm}\big)^D&=
      \big(\alpha_{nr/s}^{\pm}\big)^C 
      &&
      \textrm{if both $r$ and $s$ are even}.
%    \end{aligned}
    \label{eq:lem:invariants-pullback}
  \end{align}
  \label{lem:invariants-pullback}
\end{lemma}

\section{Examples of closed \texorpdfstring{$\Lambda_r$}{Lambda\_r}-Frobenius algebras}
\label{sec:ExamplesClosedLambdaRFrobeniusAlgebras}

In this section we present several examples of closed $\Lambda_r$-Frobenius algebras in $\Vect_{\k}$ and $\SVect_{\k}$ (Sections~\ref{sec:Ex_A}--\ref{sec:Ex_C}), which we refer to as type $A,\,B$, and~$C$.
In Section~\ref{sec:Ex_op} we define the direct sum and tensor product of closed $\Lambda_r$-Frobenius algebras.  
This in turn is applied in Section~\ref{sec:Ex_dist} to algebras of type $A,\,B$, and~$C$ to construct an explicit class of examples~$D$ of closed $\Lambda_r$-Frobenius algebras whose invariants distinguish all diffeomorphism classes of $r$-spin surfaces. 
The algebras of type~$D$ are non-semisimple, and in Section~\ref{subsec:Semisimplicity} we show that, in fact, for~$r$ not a prime, every closed $\Lambda_r$-Frobenius algebra with distinct invariants is necessarily non-semisimple.

\subsection{The closed \texorpdfstring{$\Lambda_r$}{Lambda\_r}-Frobenius algebra $A$ related to the Arf invariant}
\label{sec:Ex_A}

Let us assume that $r$ is even, and that~$\k$ is a field of characteristic~0. 
There is a closed $\Lambda_r$-Frobenius algebra $A\in\Lambda_r\!\operatorname{Frob}(\SVect_{\k})$ with 
\begin{align}
  \begin{aligned}
    A_{x}&=\k[1-x]=\k v_x\,,\quad\text{(odd for $x$ even)}\\
    \eta_{+1}(1)&= v_1\,,\quad
    \mu_{x,y} (v_x\otimes v_y)= v_{x+y-1} \,,\\
\varepsilon_{-1}(v_{-1}) &= 2\,,\quad
\Delta_{x,y}(v_{x+y+1}) = \frac{1}{2} v_x\otimes v_y \quad \text{and}\quad
  N_x(v_x)= (-1)^{1-x} v_x\,,\\
  \end{aligned}
  \label{eq:Arf-Lambda_r-FrobAlg}
\end{align}
where $\{v_x\}$ is a basis of~$A_x$. 
The corresponding $r$-spin TQFT $\Zc_{A}\colon \Bordr\lra\SVect_{\k}$ evaluated on a closed $r$-spin surface $(\Sigma_g,P,q)$ of genus~$g$ is
\begin{align}
  \begin{aligned}
    \Zc_{A}(\Sigma_g,P,q)= 2^{1-g}\cdot(-1)^{\mathrm{Arf}(\Sigma_g,P,q)}\,,
  \end{aligned}
  \label{eq:Arf-TQFT}
\end{align}
where $\mathrm{Arf}(\Sigma_g,P,q)\in\Z/2$ is the Arf invariant (see \cite{Randal:2014rs,Geiges:2012rs}, or \cite[Sect.\,3.4.2]{Szegedy:2018phd} for a review).
In particular, the invariants (introduced in Section~\ref{subsec:InvariantsFromClosedLambdaRFrobeniusAlgebras}) for~$A$ are 
\begin{align}
  \begin{aligned}
    \beta^A_d=
    \begin{cases}
      +1 & \text{if $d$ is odd}\\
      -1 & \text{if $d$ is even}\\
    \end{cases}
    \quad\text{and}\quad
    \big(\alpha_n^{\pm}\big)^A=\pm 2^{-nr/2}\,.
  \end{aligned}
  \label{eq:A-invariants}
\end{align}

\subsection{The closed \texorpdfstring{$\Lambda_r$}{Lambda\_r}-Frobenius algebra $B$ for $r>1$ odd} 
\label{sec:Ex_B}

Let $\zeta$ be a primitive $r$-th root of unity.
We define $B= \bigoplus_{x\in \Z/r} B_x\in\Vect_\k$ with 
\begin{equation}
B_0 = \operatorname{span}\big(\{w_x\}_{x\in \Z/r}\cup \{v_0\}\big) \, ,	
\quad 
B_x = \operatorname{span}\{v_x\}_{0\neq x\in \Z/r} \,.
\end{equation}
Set $u:=v_1$ and $z:=v_{-1}$. 
We define the unit by $\eta_{+1}(1)=u$, and the multiplication $\mu$ by 
\begin{align}
  \begin{aligned}
    \mu(v_x\otimes v_{-x})&=z && \text{for $x\in\Z/r$}\,,\\ 
    \mu(w_x\otimes w_{-x})&=z \quad \text{and} \quad
    \mu(w_{-x}\otimes w_{x})=\zeta^x z &&
    \text{for $0\leq x < r/2$}\,,\\ 
    \mu(u\otimes b)&=\mu(b\otimes u)=b && \text{for $b\in B$}\,,\\ 
  \end{aligned}
  \label{eq:B-profuct}
\end{align}
and 0 for any other combination of basis elements.
Moreover, we define the counit $\varepsilon \colon B_{-1}\to \k$ by $\varepsilon_{-1}(z)=1$, and the comultiplication~$\Delta$ is defined by 
\begin{align}
  \begin{aligned}
    \Delta(u) &= \sum_{x\in \Z/r} v_x\ot v_{-x} + \sum_{0\leq x<r/2}( \zeta^x w_x\ot w_{-x} + w_{-x}\ot w_x) - w_0\ot w_0\,,\\
    \Delta(v_x) &= v_x\ot z + z\ot v_x \quad\text{for } x\neq 1\,, \quad\text{and}\quad
    \Delta(w_x) = w_x\ot z + z\ot w_x\,.
  \end{aligned}
  \label{eq:B-coproduct}
\end{align}

\begin{lemma}
	\label{lem:Bfrobenius} 
	$B$ is a closed $\Lambda_r$-Frobenius algebra in $\Vect_\k$.
\end{lemma}

\begin{proof}
We verify the conditions of Lemma~\ref{lem:alt-def-clLrFa}. 
Then it is straightforward to check that the comultiplication is given by~\eqref{eq:B-coproduct}.

Associativity can be checked on basis elements. Any product of three basis elements of the form $(ab)c$ or $a(bc)$ is zero unless one of $a,b$ or $c$ is equal to $u$. But if one of the basis elements is equal to $u$ then associativity holds because $u$ is a left and a right unit. 

The non-degeneracy condition can be checked straightforwardly for the morphisms (recall the discussion preceding Lemma~\ref{lem:alt-def-clLrFa})
\begin{align}
  \delta_a(1):=&
  v_a\otimes v_{-a} \quad \text{for $a\neq0$} \,,
  \\
  \delta_0(1):=&
  v_0\otimes v_{0}
  + w_0\ot w_0 +\sum_{0< x<r/2}( \zeta^x w_x\ot w_{-x} + w_{-x}\ot w_x) \,.
  \label{eq:proof-delta-a}
\end{align}
These expressions, in turn, enable us to calculate the Nakayama automorphism~\eqref{eq:Nakayama-Ca} of~$B$. 
For any $x\in \Z/r$ we get 
\begin{align}
N_x(v_x) &= (1\ot \epsilon_{-1}\circ\mu_{x,-x})(\sigma_{B_x,B_x}\otimes 1)(v_x \otimes \delta_{x}(1)) = v_x\epsilon(v_xv_{-x}) = v_x\,,\nonumber
\\
N_0(w_x) &= (1\ot \epsilon_{-1}\circ\mu_{0,0})(\sigma_{B_x,B_x}\otimes 1)(w_x\otimes\delta_{0}(1)) \nonumber 
\\
& = \sum_{0\leq y<r/2} \big(\zeta^yw_y \epsilon(w_xw_{-y}) + w_{-y}\epsilon(w_xw_y)\big)\,.
\end{align}
Hence for both $0\leq x<r/2$ and $r/2<x<r$ we find $N_0(w_x) =\zeta^xw_x$.  

The commutativity axiom can be verified directly, using the fact that most of the products of basis vectors are zero. 
The only non-trivial part of the calculation is that for $0\leq x<r/2$ we have 
$ 
w_{-x}w_x = \zeta^xz = N_0^{-1}(w_x)w_{-x} = w_xN_0(w_{-x})
$. 

For the twist axiom we first compute $f(x,k):=\mu_{x,-x}(N_x^k\ot 1)\delta_{x}(1)$. 
For $x\neq 0$ we get $f(x,k) = \mu_{x,-x}(N_x^k\ot 1)(v_x\ot v_{-x}) = z$, while for $x=0$ we get 
\begin{align}
f(0,k)
	& = 
	\mu_{0,0}(N_0^k) \Big(v_0\ot v_0 + \sum_{0\leq x<r/2} (\zeta^{x}w_x\ot w_{-x} + w_{-x}\ot w_x) - w_0\ot w_0\Big) \nonumber 
	\\
	& = 
	z + \sum_{0\leq x<r/2} \big( \zeta^{x+kx}z + \zeta^{-x-kx}z - z \big) \nonumber 
	\\ 
	& = 
	z \Big( \sum_{x\in \Z/r} \zeta^{(1+k)x} + 1 \Big) 
	= 
	\begin{cases} z & \text{ if } k\neq -1\\ (r+1)z & \text{ if } k=-1\end{cases} \, . 
\end{align}
We need to show that $f(x,k) = f(x-k-1,k)$ for every $x,k\in \Z/r$. 
If $k\neq -1$ then both $f(x,k) = z$ and $f(x-k-1,k) = z$. 
If $k=-1$ then $f(x-(-1)-1,-1) = f(x,-1)$, 
and we are done. 
\end{proof}

Finally we compute invariants (recall Section~\ref{subsec:InvariantsFromClosedLambdaRFrobeniusAlgebras}). 
\begin{lemma}\label{lem:roddinvariants}
 The invariants of the closed $\Lambda_r$-Frobenius algebra  $B$ are
 $$\beta_d^B  = \begin{cases} 1 & \textrm{if $d\neq r$,}\\ r+1 & \text{if $d=r$,}\end{cases}\quad\textrm{and}\quad \alpha_n^B = 0\,.$$
 \end{lemma}
 \begin{proof}
   From Equation \ref{eq:betaInvariants} we know that $\beta_d = \dim(B_d)$, and the first part of the lemma already follows from the definition of $B$. 
   For the second part, recall the operators~$h_x$ and$~H_x$ in\eqref{eq:HandleMap} and~\eqref{eq:Hx}, respectively, and the definition $\alpha_n = \Tr(H_x^n)$ from~\eqref{eq:alphaOdd}. 
   However, if $x\neq 1$ then $m_{a,b}\Delta_{a,b}(v_x)=0$, and in particular $h_x=0$ and $H_x=0$.
\end{proof}

\begin{notation}
	To stress the value of~$r$, we denote the algebra constructed above by $B^{(r)}$. 
\end{notation}

\subsection{The closed \texorpdfstring{$\Lambda_r$}{Lambda\_r}-Frobenius algebra $C$ for $r>2$ even} 
\label{sec:Ex_C}

Let $\zeta$ again be a primitive $r$-th root of unity.
We define $C= \bigoplus_{x\in \Z/r} C_x\in\Vect_\k$ with 
\begin{equation}
C_0 = \operatorname{span}\big( \{w_{i,j}\}_{i\in\{0,1\},j\in \Z/r}\cup\{v_0\} \big) \,, 
\quad 
C_x = \operatorname{span}\{v_x\}_{0\neq x\in \Z/r} \,.
\end{equation}
Write $u:=v_1$ and $z:=v_{-1}$. 
We define the unit by $\eta_{+1}(1)=u$, and the multiplication $\mu$ by 
\begin{align}
\begin{aligned}
\mu(v_x\otimes v_{-x})&=z && \text{for $x\in\Z/r$}\,,\\ 
\mu(w_{0,x}\otimes w_{1,-x})&=z \quad \text{and} \quad
\mu(w_{1,-x}\otimes w_{0,x})=\zeta^{-x} z &&
\text{for $x\in\Z/r$}\,,\\ 
\mu(u\otimes c)&=\mu(c\otimes u)=c && \text{for $c\in C$}\,,\\ 
\end{aligned}
\label{eq:C-product}
\end{align}
and 0 for any other combination of basis elements.
We define the counit by $\varepsilon_{-1}(z)=1$,
and the comultiplication~$\Delta$ acts non-trivially as follows: 
\begin{align}
\begin{aligned}
\Delta(u) &= \sum_{x\in \Z/r} \Big(v_x\ot v_{-x} + \zeta^x w_{0,x}\ot w_{1,-x} + w_{1,-x}\ot w_{0,x}\Big)\,,\\
\Delta(v_x) &= v_x\ot z + z\ot v_x \quad\text{for } x\neq 1\,, \\
\Delta(w_{0,x}) &= w_{0,x}\ot z + z\ot w_{0,x} \quad\text{and}\quad
\Delta(w_{1,x}) = w_{1,-x}\ot z + z\ot w_{1,-x}\,.
\end{aligned}
\label{eq:C-coproduct}
\end{align}

\begin{lemma} 
	\label{lem:Cfrobenius}
	$C$ is a closed $\Lambda_r$-Frobenius algebra in $\Vect_\k$.
\end{lemma}
\begin{proof} 
	Associativity and unitality can be checked analogously to the case of~$B$ in Section~\ref{sec:Ex_B}.
	The non-degeneracy condition can be shown for the morphisms
	\begin{align}
	\delta_a(1):=&
	v_a\otimes v_{-a} \quad \text{for $a\neq0$ and}\\
	\delta_0(1):=&
	v_0\otimes v_{0}
	+ \sum_{x\in\Z/r}\Big( \zeta^x w_{0,x}\ot w_{1,-x} + w_{1,-x}\ot w_{0,x}\Big) \,,
	\label{eq:C-proof-delta-a}
	\end{align}
	which can be checked to correspond to~$\Delta$ in~\eqref{eq:C-coproduct}. 
	
	We calculate the Nakayama automorphism of $C$ analogously to the case of~$B$. 
	For $x\in \Z/r$ we find 
	\begin{align}
	N_x(v_x) = v_x\,, \quad
	N_0(w_{0,x}) = \zeta^x w_{0,x}\,, \quad
	N_0(w_{1,x}) = \zeta^{x} w_{1,x}\,.
	\label{eq:C-Nakayama}
	\end{align}
	
	For the commutativity axiom, we only need to take care of the multiplications that are non-zero. We use the formulas for the Nakayama automorphism and get
	\begin{align}
	\begin{aligned}
	v_x v_{-x} &= N_x^{x-1}(v_{-x})v_x = v_{-x}N_{-x}^{1-x}(v_x) = z\,,\\
	w_{0,x}w_{1,-x} &= N_0^{-1}(w_{1,-x})w_{0,x} = w_{1,-x}N_0(w_{0,x}) = z  \,,\\
	w_{1,-x}w_{0,x} &= N_0^{-1}(w_{0,x})w_{1,-x} = w_{0,x}N_0(w_{1,-x}) = \zeta^{-x}z\,.
	\end{aligned}
	\end{align}
	
	For the twist axiom, we write as before $f(x,k):=m_{x,-x}(N_x^k\ot 1)\Delta_{x,-x}(u)$. 
	For $x\neq 0$ we have 
	$f(x,k) = \mu_{x,-x}(N_x^k\ot 1)(v_x\ot v_{-x}) = z$. 
	For $x=0$ we have
	\begin{align}
	\begin{aligned}
	f(0,k) &= \mu_{0,0}(N_0^k\otimes 1)\Big(v_0\ot v_0 + \sum_{x\in \Z/r} (w_{1,-x}\ot w_{0,x} + \zeta^x w_{0,x}\ot w_{1,-x}) \Big) \\
	&=z + \sum_{x\in \Z/r} \zeta^{-x-kx}z + \zeta^{x+kx}z = \Big(2\sum_{x\in \Z/r} \zeta^{(1+k)x} + 1\Big)z \\
	&= 
	\begin{cases} 
	z & \text{if  $k\neq -1$,}\\ 
	(2r+1)z & \text{if  $k=-1$.}
	\end{cases}
	\end{aligned}
	\end{align}
	We need to show that $f(x,k) = f(x-k-1,k)$ for every $x,k\in \Z/r$. If $k=-1$, then this equality holds trivially, and if $k\neq -1$ then $f(x,-1) = f(x-(-1)-1,-1)$. 
\end{proof}

The proof of Lemma~\ref{lem:roddinvariants} generalises directly to a proof of the following:
\begin{lemma} 
	\label{lem:C-invariants}
	The invariants of the closed $\Lambda_r$-Frobenius algebra $C$ are 
	\begin{align}
	\beta_d^C = 
	\begin{cases} 
	1 &\textrm{if $d\neq r$,} \\ 
	2r+1& \textrm{if $d=r$,}
	\end{cases} 
	\quad
	\text{and}
	\quad
	\big(\alpha_n^{\pm}\big)^C = 0\,.
	\label{eq:lem:C-invariants}
	\end{align}
\end{lemma}

\begin{notation}
To stress the value of~$r$, we denote the algebra constructed above by $C^{(r)}$. 
\end{notation}

\subsection{A one-parameter family of 1-dimensional algebras}

Let $\kappa\in \k^{\times}$, and consider the 1-dimensional algebra~$\k$. 
We define a counit by $\epsilon(1) = \kappa^{-1}$. 
This gives us a commutative Frobenius algebra, where $\Delta(1) = \kappa 1\ot 1$. The handle endomorphism is then multiplication by $\kappa$. 
Pulling back with $P_{r,1}^*$ from~\eqref{eq:PullbackRS}, we get a $\Lambda_r$-Frobenius algebra which we denote by $E_{\kappa}$. 

\begin{lemma}
	\label{lem:E}
If $r$ is odd then the invariants of $E_{\kappa}$ are 
\begin{equation}
\beta_d^{E_{\kappa}} = 1
	\, , \quad 
	\alpha_n^{E_{\kappa}} = \kappa^{rn}
	\qquad 
	\textrm{for all $d|r$ and $n\in\Z_{\geqslant 1}$.} 
\end{equation}
If $r$ is even then the invariants of $E_{\kappa}$ are 
\begin{equation}
\beta_d^{E_{\kappa}}=1
	\, , \quad 
	(\alpha_n^{+})^{E_{\kappa}} =(\alpha_n^{-})^{E_{\kappa}} = 
	  \kappa^{rn/2}
	\qquad 
	\textrm{for all $d|r$ and $n\in\Z_{\geqslant 1}$.} 
\end{equation}
\end{lemma}

\begin{proof}
Using Lemma~\ref{lem:invariants-pullback}, we may forget about the $r$-spin structure, and use the fact that the handle endomorphism $H$ is multiplication by $\kappa$, and therefore the trace of $H^{rn}$ is $\kappa^{rn}$. 
\end{proof}

\subsection{Another non-semisimple algebra}

Consider the algebra $\k[x]/(x^2)$. 
This is a commutative Frobenius algebra with $\epsilon(1)=0$ and $\epsilon(x)=1$. 
We calculate the comultiplication as before to get $\Delta(1) = 1\ot x + x\ot 1$ and $\Delta(x) = x\ot x$.  
The handle endomorphism is then given by $1\mapsto 2x$, $x\mapsto 0$, and is therefore nilpotent. 
We denote the pullback of this algebra to a closed $\Lambda_r$-Frobenius algebra by $F$. 

\begin{lemma}
If $r$ is odd then the invariants of $F$ are 
\begin{equation}
\beta_d^F = 2
	\, , \quad 
	\alpha_n^F =0
	\qquad 
	\textrm{for all $d|r$ and $n\in\Z_{\geqslant 1}$.} 
\end{equation}
If $r$ is even then the invariants of $F$ are 
\begin{equation}
\beta_d^F = 2
	\, , \quad 
	(\alpha_n^{+})^F =(\alpha_n^{-})^F
	  =0
	\qquad 
	\textrm{for all $d|r$ and $n\in\Z_{\geqslant 1}$.} 
\end{equation}
\end{lemma}

\begin{proof} This follows exactly the same reasoning as the proof of Lemma~\ref{lem:E}, using the fact that the handle endomorphism is nilpotent. 
\end{proof}

\subsection{Operations on algebras}
\label{sec:Ex_op}

Let $X = \bigoplus_{x\in \Z/r}X_x$ and $Y= \bigoplus_{x\in \Z/r}Y_x$ be two closed $\Lambda_r$-Frobenius algebras in $\Vect_\k$, with respective structure maps $(\mu^X_{x,y},\Delta^X_{x,y},\eta^X_1, \epsilon^X_{-1})$ and $(\mu^Y_{x,y},\Delta^Y_{x,y},\eta^Y_1, \epsilon^Y_{-1})$. 
We set 
\begin{equation}
X\oplus Y := \bigoplus_{x\in \Z/r} (X_x\oplus Y_x) \, , 
\qquad
X\otl Y := \bigoplus_{x\in \Z/r} (X_x\ot_\k Y_x)\,.
\end{equation}
Notice that in general $X\otl Y$ is a proper subspace of $X\ot Y$. 
Both $X\oplus Y$ and $X\otl Y$ naturally inherit the structure of a closed $\Lambda_r$-Frobenius algebra from~$X$ and~$Y$: 

\begin{lemma}\label{lem:plustimes}
	\begin{enumerate}[label={(\roman*)}]
		\item 
		$(X\oplus Y, \mu^X_{x,y}\oplus \mu^Y_{x,y},\Delta^X_{x,y}\oplus \Delta^Y_{x,y},\eta^X_1\oplus \eta^Y_1,\epsilon^X_{-1}\oplus \epsilon^Y_{-1}) \in \Lambda_r\!\operatorname{Frob}(\Vect_\k)$. 
		\item 
		$(X\otl Y, \mu^X_{x,y}\ot \mu^Y_{x,y}, \Delta^X_{x,y}\ot \Delta^Y_{x,y},\eta^X_1\ot \eta^Y_1,\epsilon^X_{-1}\ot \epsilon^Y_{-1}) \in \Lambda_r\!\operatorname{Frob}(\Vect_\k)$. 
		\item 
		\label{item:InvariantsDistribute}
		If~$\Sigma$ is a closed connected $r$-spin surface, then $\Zc_{X\oplus Y}(\Sigma) = \Zc_X(\Sigma) + \Zc_Y(\Sigma)$ and $\Zc_{X\otl Y}(\Sigma) = \Zc_X(\Sigma) \Zc_Y(\Sigma)$.
	\end{enumerate}
	\label{lem:dsum-tprod}
\end{lemma}

\begin{proof}
	The fact that $X\oplus Y$ and $X\otl Y$ both satisfy the axioms of a closed $\Lambda_r$-Frobenius algebra is a direct verification, and is based on the fact that all the axioms are given by equalities between connected diagrams. 
	
	For part~\ref{item:InvariantsDistribute} about the invariants, we first notice that $N^{X\oplus Y}_x\colon X_x\oplus Y_x\to X_x\oplus Y_x$ is given by $N^X_x\oplus N^Y_x$. 
	This is a straightforward verification that follows from the definition of the Nakayama automorphism. 
	Similarly, $N^{X\otl Y}_x\colon X_x\ot Y_x\to X_x\ot Y_x$ is equal to $N^X_x\ot N^Y_x$. 
	If $\Sigma$ is a surface of genus 1 then there is a divisor~$d$ of~$r$ such that the invariant associated to~$\Sigma$ by any given closed $\Lambda_r$-Frobenius algebra is $\epsilon \mu_{0,0}(N^{d-1}\ot 1) \Delta_{0,0}(u)$. 
	Using the definition of structure maps we get  
	\begin{align}
	\begin{aligned}
	\Zc_{X\oplus Y}(\Sigma) &= \epsilon^{X\oplus Y}\mu_{0,0}^{X\oplus Y}\big((N^{X\oplus Y})^{d-1}\ot 1\big) \Delta_{0,0}^{X\oplus Y}u^{X\oplus Y} \\
	&= 
	\epsilon^X \mu_{0,0}^X\big((N^X)^{d-1}\ot 1\big) \Delta_{0,0}^Xu^Y + \epsilon^Y \mu_{0,0}^Y\big((N^Y)^{d-1}\ot 1\big) \Delta_{0,0}^Yu^Y\\
	&= \Zc_X(\Sigma) + \Zc_Y(\Sigma)\,.
	\end{aligned}
	\end{align}
	A similar calculation shows that $\Zc_{X\otl Y}(\Sigma) = \Zc_X(\Sigma)\Zc_Y(\Sigma)$. 
	If the genus of $\Sigma$ is larger than 1, we use the fact that $\Zc(\Sigma) = \Tr(H_x^n)$ for some $n>0$ if $r$ is odd, and of the form $\Tr((H_x^+)^n)$ or $\Tr((H_x^+)^{n-1} H_x^-)$ for some $n>0$ if $r$ is even (recall~\eqref{eq:Hx}--\eqref{eq:alphaEvenMinus}). 
	One computes that $H_x^{X\oplus Y} = H_x^X\oplus H_x^Y$ and $H_x^{X\otl Y} = H_x^X\ot H_x^Y$ for~$r$ odd, and similar identities hold when~$r$ is even. 
\end{proof}

\subsection {A closed \texorpdfstring{$\Lambda_r$}{Lambda\_r}-Frobenius algebra with distinct torus invariants}
\label{sec:Ex_dist}

First let us assume that $r>1$ is odd and has prime factorisation 
\begin{equation}
r=p_1^{l_1}\cdots p_k^{l_k}
\end{equation} 
with increasingly ordered primes $2<p_1 < \dots < p_k$ and exponents $l_1,\dots,l_k>0$.
Consider the closed $\Lambda_r$-Frobenius algebra
\begin{align}
D:=\bigoplus_{i=1}^{k}(D^{(p_i)})^{\oplus c_i}\,,
\quad\text{where}\quad
D^{(p_i)} & =\bigoplus_{m=1}^{l_i}P_{r,p_i^m}^*(B^{(p_i^m)})\,,
\\
c_1 & =1 \, , 
\\
  c_{i+1} & =c_i\cdot(\beta_r^{D^{(p_i)}}+1)
\label{eq:def:D}
\end{align}
and the algebras $B^{(p_i^m)}$ are as defined in Section~\ref{sec:Ex_B}.

\begin{proposition}
	The torus invariants of $D$ %in \eqref{eq:def:D} 
	are pairwise distinct.
	\label{prop:inv-D-distinct}
\end{proposition}
\begin{proof}
	We start with computing the invariants of $D^{(p_i)}$.
	First, by Lemma~\ref{lem:invariants-pullback} we have
	\begin{align}
	\beta_d^{P_{r,p_i^m}^*(B^{(p_i^m)})}=
	\begin{cases}
	1 & \text{if $p_i^m\nmid d$,}\\
	1+p_i^m & \text{if $p_i^m | d$.}
	\end{cases}
	\label{eq:inv-B-pb}
	\end{align}
	Applying Lemma~\ref{lem:dsum-tprod}\ref{item:InvariantsDistribute} to Lemma~\ref{lem:roddinvariants}, this gives us 
	\begin{align}
	\beta_d^{D^{(p_i)}}=
	\sum_{\substack{1\le m\le l_i \\ p_i^m \nmid d}}1
	+\sum_{\substack{1\le m\le l_i \\ p_i^m | d}}(1+p_i^m)
	= l_i
	+\sum_{\substack{1\le m\le l_i \\ p_i^m | d}}p_i^m \,, 
	\label{eq:inv-D-factor}
	\end{align}
	and hence 
	\begin{align}
	\beta_d^{D}
	= \sum_{1\le i\le k} 
	c_i\beta_d^{D^{(p_i)}}
	= \sum_{1\le i\le k} c_i\Big(
	l_i +\sum_{\substack{1\le m\le l_i \\ p_i^m | d}}p_i^m \Big)\,.
	\label{eq:inv-D-full}
	\end{align}
	
	We need to show that if $d$ and $e$ are distinct divisors of~$r$, then $\beta_d^D\neq \beta_e^D$. 
	Write $d = p_1^{a_1}\cdots p_k^{a_k}$ and $e=p_1^{b_1}\cdots p_k^{b_k}$. Let $i$ be the minimal index such that $a_i\neq b_i$. Without loss of generality, we assume that $a_i<b_i$.
	It holds that $\beta_d^{D^{(p_j)}} = \beta_e^{D^{(p_j)}}$ for $j<i$. 
	From~\eqref{eq:inv-D-full} we get 
	\begin{equation}
	\beta_d^D-\beta_e^D 
	= 
	c_i\big(\beta_d^{D^{(p_i)}} - \beta_e^{D^{(p_i)}}\big) + \sum_{j=i+1}^k c_j\big(\beta_d^{D^{(p_j)}}-\beta_e^{D^{(p_j)}} \big)
	\end{equation}
	By definition of the numbers $c_i$ it holds that $c_i|c_j$ for $i<j$. Reducing the above equation modulo $c_{i+1}$ gives 
	\begin{equation}
	\beta_d^D-\beta_e^D = c_i\big(\beta_d^{D^{(p_i)}} - \beta_e^{D^{(p_i)}}\big) \text{ mod } c_{i+1}
	\end{equation}
	Dividing by~$c_i$ and using again the definition of $c_{i+1}$ and~\eqref{eq:inv-D-factor}, we get
	\begin{equation} 
	\frac{1}{c_i}\big(\beta_d^D-\beta_e^D\big) = \sum_{a_i< k\leq b_i}p_i^k \quad \text{mod } \beta_r^{D^{(p_i)}}+1 \,.
	\end{equation} 
	Since the number $\sum_{a_i< k\leq b_i}p_i^k$ is a positive integer that is less than $\beta_r^{D^{(p_i)}}+1$, it holds that the difference $\beta_d^D-\beta_e^D$ is non-zero as well, and we are done. 
\end{proof}

In case $r$ is even, one just uses the algebras $C$ instead of $B$. The proof that the resulting algebra has distinct torus invariants is similar. 
Together with Theorem~\ref{thm:classification} and the discussion in Section~\ref{subsec:InvariantsFromClosedLambdaRFrobeniusAlgebras}, this proves Theorem~\ref{thm:semimain}.

\subsection{Semisimplicity and torus invariants}
\label{subsec:Semisimplicity}

The algebras $B^{(r)}$ and $C^{(r)}$ that are used to provide distinct invariants of $r$-spin tori are both non-semisimple.
In this section we show that for a semisimple algebra the torus invariants can have at most two distinct values.
We assume that the field~$\k$ is algebraically closed. 

\medskip 

We call a closed $\Lambda_r$-Frobenius algebra $D=\bigoplus_{x\in\Zb/r}D_x$ \textsl{semisimple} if $D$ as an algebra with product $\mu=\sum_{x,y\in\Zb/r}\mu_{x,y}$ and unit $\eta=\eta_{+1}$ is semisimple.
Consider a closed $\Lambda_r$-Frobenius algebra $D$ and define a new grading by $D'_a:=D_{1-a}$.
By setting $\mu'_{a,b}:=\mu_{1-a,1-b}\colon D_{1-a}\otimes D_{1-b}=D'_{a}\otimes D'_{b}\lra D'_{a+b}=D_{1-a-b}=D_{1-a+1-b-1}$ and $\eta':=\eta_{+1}\colon \k \lra D_{+1}=D'_{0}$ we obtain a $\Zb/r$-graded algebra.
Note that the counit with this new grading $\varepsilon':=\varepsilon_{-1}\colon D_{-1}=D'_{0}\lra \k$ is of degree $-2$.

\begin{theorem}
  Let $D$ be a semisimple closed $\Lambda_r$-Frobenius algebra in $\Vect_\k$ or in $\SVect_\k$. 
  Then for all $a\in\Zb/r$ we have
  \begin{align}
    \begin{aligned}
      D_a\cong
      \begin{cases}
	D_1&\text{if $r$ is odd,}\\
	D_1&\text{if $r$ is even and $a$ is odd,}\\
	D_0&\text{if $r$ is even and $a$ is even.}\\
      \end{cases}
    \end{aligned}
    \label{eq:thm:semisimple-clLrFrob-alg}
  \end{align}
  \label{thm:semisimple-clLrFrob-alg}
\end{theorem}
\begin{proof}
  In both cases $D$ is a direct sum of matrix algebras.
  By \cite[Lem.\,2.2.8]{Kockbook} the counit $\varepsilon'$ on $D'$ is of the form
  $\varepsilon=\Tr(X\cdot-)$, where $X\in D'$ is invertible, and $\Tr$ is the trace on the regular representation.
  Since $\varepsilon'$ is of degree $-2$, $X$ is of degree $+2$, i.e.\ we have isomorphisms $\mu(X\otimes -)\colon D'_a\lra D'_{a+2}$ for every $a\in\Zb/r$.

  In particular we have $D'_a\cong D'_0$ for every $a\in\Zb/r$ if $r$ is odd.
  If $r$ is even, we have 
  $D'_a\cong D'_0$ if $a$ is even and
  $D'_a\cong D'_1$ if $a$ is odd.
\end{proof}

\begin{corollary}
  The torus invariants of a semisimple closed $\Lambda_r$-Frobenius algebra in $\Vect_\k$ or in $\SVect_\k$ can take at most two distinct values.
  \label{cor:ssi-invariants}
\end{corollary}

\section{Symmetric monoidal categories, invariants, and TQFTs}
\label{sec:symmcat}

In this section we answer the question about $r$-spin invariants stated in the introduction, and prove Theorem \ref{thm:main} by using symmetric monoidal categories that are freely generated by an algebraic structure.

\subsection{Basic notions}
\label{subsec:BasicNotions}

We start by recalling the relevant definitions and results from \cite{meir21}.
Let $\C$ be a symmetric monoidal category which we further assume to be $\k$-linear and rigid (hence in particular all Hom sets are $\k$-vector spaces, and every object has a chosen dual).
An \textsl{algebraic structure} in $\C$ consists of an object $W \in \C$ together with a choice of morphisms, called \textsl{structure tensors}, $x_i\colon W^{\ot q_i}\to W^{\ot p_i}$ for non-negative integers $q_i,p_i$. 
We call the tuple $((p_i,q_i))$ the \textsl{type} of $(W,(x_i))$. 
If $(x_i)$ are clear from the context, we write~$W$ instead of $(W,(x_i))$.

For example, algebraic structures of type $((1,1)^k)$ are just representations of the free algebra on $k$ generators. Algebraic structures of type $((1,2))$ include Lie algebras and associative algebras. 
Algebraic structures of type $((1,2),(2,1),(0,1),(1,0))$ include bialgebras and Frobenius algebras (with structure maps $\Delta,\mu,\eta,\epsilon$, respectively). 
A Hopf algebra is an algebraic structure of type $((1,2),(2,1),(0,1),(1,0),(1,1))$, where the extra structure tensor corresponds to the antipode. 

A closed $\Lambda_r$-Frobenius algebra is an algebraic structure with underlying object~$W$ and tensors $(P_0,\ldots, P_{r-1},\mu,\Delta,\epsilon,\eta)$ where, for $i \in \{0,\ldots, r-1\}$, $P_i\colon W\to W$ is an idempotent and $W=\bigoplus_i \operatorname{Im}(P_i)$. 
We write $W_i = \operatorname{Im}(P_i)$. 
The structure tensor~$\mu$ then stands for $\sum_{i,j,k}\mu_{i,j}^k$, and $\Delta$ is interpreted similarly. 
To avoid cumbersome notation we often just write the collection of our structure tensors as $(x_i)$ and denote the type of our algebraic structure by $((p_i,q_i))$.  

A \textsl{diagram} contains boxes and strings. Every box in the diagram is labelled by one of the structure tensors or $1_W$. A box that is labelled by $1_W$ has one input string and one output string, and a box labelled by the structure tensor $x_i$ has $q_i$ input strings and $p_i$ output strings. 
Some of the input strings may be connected to some output strings, and strings may be permuted. 
We think of diagrams with $q$ input strings and $p$ output strings as representing a morphism $W^{\ot q}\to W^{\ot p}$. 
An example with $q=3=p$ is 
\begin{equation}\scalebox{0.7}{
\begin{tikzpicture}[baseline={(current  bounding  box.center)}]
	\begin{pgfonlayer}{nodelayer}
		\node [style=none] (3) at (-6.75, 3.5) {};
		\node [style=none] (4) at (-7, 2.5) {};
		\node [style=none] (5) at (-6.5, 2.5) {};
		\node [style=none] (6) at (-3.25, 2.5) {};
		\node [style=none] (7) at (-5.25, 3.5) {};
		\node [style=none] (8) at (-4.75, 3.5) {};
		\node [style=none] (9) at (-3.25, 3.5) {};
		\node [style=none] (10) at (-5, 2.5) {};
		\node [style=none] (11) at (-6.5, 1.25) {};
		\node [style=none] (12) at (-7, 1.25) {};
		\node [style=none] (13) at (-5.25, 4.25) {};
		\node [style=none] (14) at (-4.75, 4.25) {};
		\node [style=none] (15) at (-5, 1.25) {};
		\node [style=none] (16) at (-3.25, 4.25) {};
		\node [style=none] (17) at (-3.75, 1.75) {};
		\node [style=none] (18) at (-6.75, 4.25) {};
		\node [style=1function] (19) at (-3.25, 3) {$x_3$};
		\node [style=2function] (20) at (-5, 3) {$x_2$};
		\node [style=2function] (21) at (-6.75, 3) {$x_1$};
		\node [style=none] (22) at (-2, 4.25) {};
		\node [style=none] (23) at (-2, 1.75) {};
	\end{pgfonlayer}
	\begin{pgfonlayer}{edgelayer}
		\draw [in=90, out=-90] (14.center) to (7.center);
		\draw [in=90, out=-90] (13.center) to (8.center);
		\draw (18.center) to (3.center);
		\draw [in=90, out=-90, looseness=1.25] (5.center) to (12.center);
		\draw [in=-90, out=90] (11.center) to (4.center);
		\draw [in=75, out=-105, looseness=0.75] (6.center) to (15.center);
		\draw (9.center) to (16.center);
		\draw [in=105, out=-90] (10.center) to (17.center);
		\draw [bend left=270, looseness=1.50] (22.center) to (16.center);
		\draw (22.center) to (23.center);
		\draw [in=-75, out=-90, looseness=1.50] (23.center) to (17.center);
	\end{pgfonlayer}
\end{tikzpicture}
} \, . 
\end{equation}
Two diagrams are equivalent if and only if they define the same map for every algebraic structure of type $((p_i,q_i))$, cf.\ \cite[Def.\,4.1]{meir20}.
We define $\operatorname{Con}^{p,q}$ to be the vector space freely spanned by diagrams with $q$ input strings and $p$ output strings. 
In particular, $\operatorname{Con}^{0,0}$ is freely spanned by closed diagrams, and it has the structure of a commutative $\k$-algebra, where the multiplication is given by juxtaposition of diagrams. 
We write $\winf \equiv \winf((p_i,q_i))$ for the algebra $\operatorname{Con}^{0,0}$. 

In \cite{meir21} the universal category $\C_{\textrm{univ}}$ generated by a dualisable object~$W$ with structure tensors of type $((p_i,q_i))$ was constructed. 
The objects in $\C_{\textrm{univ}}$ are formal direct sums of objects of the form $W^{a,b}:= W^{\ot a}\ot (W^*)^{\ot b}$, and the Hom spaces are defined using the spaces $\operatorname{Con}^{p,q}$. 
The category $\C_{\textrm{univ}}$ satisfies the following universal property:

\begin{proposition}\cite[Prop.\,3.1]{meir21} 
	Let $\D$ be a symmetric monoidal category. 
	Evaluation at $W$ gives a bijection between isomorphism classes of symmetric monoidal functors $\C_{\textrm{univ}}\to \D$ and isomorphism classes of algebraic structures of type $((p_i,q_i))$ in $\D$.
\end{proposition}

We would like to consider here only algebraic structures that satisfy the axioms of a closed $\Lambda_r$-Frobenius algebra. All the axioms we consider can be thought of as the vanishing of certain elements in $\operatorname{Con}^{p,q}$ (for some $p,q\in\N$). A \textsl{theory} is a collection of axioms $\T\subseteq \bigsqcup_{p,q}\operatorname{Con}^{p,q}$. 
An algebraic structure of type $((p_i,q_i))$ is said to be a \textsl{model} of $\T$ if every element of $\T$ is zero when considered as a map between tensor powers of $W$. We consider here the theory $\T$ of closed $\Lambda_r$-Frobenius algebras. 
In \cite[Sect.\,3.1]{meir21} a quotient $\C_{\textrm{univ}}^{\T}$ of $\C_{\textrm{univ}}$ was constructed which has the following universal property:

\begin{proposition}\label{prop:univ2}\cite[Prop.\,3.5]{meir21} 
	Let $\D$ be a symmetric monoidal category. Evaluation at $W$ gives a bijection between isomorphism classes of symmetric monoidal functors $\C_{\textrm{univ}}^{\T}\to \D$ and isomorphism classes of closed $\Lambda_r$-Frobenius algebras in $\D$. 
\end{proposition}

In particular, since $\C_{\textrm{univ}}^{\T}$ is a quotient of $\C_{\textrm{univ}}$, the endomorphism ring of $\one$ in $\C_{\textrm{univ}}^{\T}$ is a quotient of $\winf$, which we denote by $\winf/I_{\T}=\winf^{\T}$. 

Define now a symmetric monoidal category $\E_0$ in the following way: the objects of $\E_0$ are the same as the objects of $\Bord_2^r$, and for $X,Y\in\E_0$ we have 
\begin{equation}
\Hom_{\E_0}(X,Y) = \k\Hom_{\Bord_2^r}(X,Y) \, , 
\end{equation}
the free $\k$-vector space generated by the Hom set in $\Bord_2^r$. 
Composition is defined in the obvious way. The symmetric monoidal structure on $\Bord_2^r$ induces a symmetric monoidal structure on $\E_0$. 
Take now $\E$ to be the idempotent completion of the additive closure of $\E_0$. 
Hence objects in $\E$ are direct summands of objects of the form $\bigoplus_{i=1}^nX_i$, where $X_i\in \E_0$, and 
\begin{equation}
\Hom_{\E}\Big(\bigoplus_i X_i,\bigoplus_j Y_j\Big) 
	= 
	\bigoplus_{i,j} \Hom_{\E_0}(X_i,Y_j) \, . 
\end{equation}
Hom spaces between direct summands are defined in the obvious way. 

\begin{proposition}
There is a natural equivalence of symmetric monoidal $\k$-linear categories between the category $\C_{\textrm{univ}}^{\T}$ and $\E$. 
\end{proposition}

\begin{proof}
It is straightforward to see that if $\D$ is a $\k$-linear idempotent complete rigid symmetric monoidal category, then the category of symmetric monoidal functors $\Bord_2^r\to \D$ is equivalent to the category of symmetric monoidal $\k$-linear functors  $\E\to \D$. 
By Theorem~\ref{thm:closed-r-spin-TQFT-classification} we know that functors $\Bord_2^r\to\D$ are equivalent to closed $\Lambda_r$-Frobenius algebras in $\D$. Proposition \ref{prop:univ2} above shows now that we have an equivalence of categories between functors from $\E$ to $\D$ and functors from $\C_{\textrm{univ}}^{\T}$ to $\D$. By taking $\D=\E$ we get the desired equivalence. 
\end{proof}

It follows from the construction of~$\E$ that the endomorphism ring of the unit object is the polynomial algebra on the set of all diffeomorphism classes of closed connected surfaces with an $r$-spin structure. 
By the equivalence above this algebra is isomorphic to $\winf^{\T}$. 
Hence we have the identifications 
\begin{align}
\winf^{\T} 
	& = 
	\k[\{\beta_d\}_{d|r}, \alpha_1,\alpha_2,\ldots]
	&&\text{ if } r \text{ is odd, }
	\\
	\winf^{\T} 
	& = 
	\k[\{\beta_d\}_{d|r}, \alpha_1^+,\alpha_1^-,\alpha_2^+,\alpha^{-}_2,\ldots]
	&&\text{ if } r \text{ is even} 	
\end{align}
for $r>2$ (and we have to add a variable $\alpha_{-1}$ for the sphere invariant in case $r\in\{1,2\}$). 

Every closed $\Lambda_r$-Frobenius algebra in $\D$ gives rise to a functor $\C_{\textrm{univ}}^{\T} \lra \D$, and hence in particular to a ring homomorphism $\End_{\C_{\textrm{univ}}^{\T}}(\one)\to \End_{\D}(\one)$, that is, a map $\chi\colon\winf^{\T}\to \End_{\D}(\one)$. 
If $\End_{\D}(\one)=\k$ we get $\chi\colon\winf^{\T}\to \k$. 
We call $\chi$ the \textsl{character of invariants} of the closed $\Lambda_r$-Frobenius algebra. 
We will think of such a $\chi$ as a function that assigns a numerical value to every diffeomorphism class of $r$-spin surfaces. 
Since $\winf^{\T}$ is the polynomial algebra on the isomorphism classes of closed connected surfaces with an $r$-spin structure, such a character gives the scalar invariants that the TQFT corresponding to the closed $\Lambda_r$-Frobenius algebra associates to such surfaces.

\subsection{The categories \texorpdfstring{$\C_{\chi}$}{C\_chi}}
\label{subsec:CategoriesCchi}

Assume now that~$\chi$ is a character as above, i.e.\ a ring homomorphism $\winf^{\T}\to \k$. 
To any two objects $X,Y\in \C_{\textrm{univ}}^{\T}$ we associate the pairing 
\begin{align}
\langle -,-\rangle\colon \Hom_{\C_{\textrm{univ}}^{\T}}(X,Y) \ot \Hom_{\C_{\textrm{univ}}^{\T}}(Y,X)
	& \to 
	\End_{\C_{\textrm{univ}}^{\T}}(\one)=\winf^{\T} 
	\nonumber
	\\
	T_1 \ot T_2
	& \mapsto 
	\Tr(T_1\circ T_2) \, ,
\end{align}
and we define 
\begin{equation}
\langle -,-\rangle_{\chi}:= \chi\circ \langle -,-\rangle\, .
\end{equation}

\begin{definition}
A morphism $f\colon X\to Y$ in $\C_{\textrm{univ}}^{\T}$ is called \textsl{$\chi$-negligible} if it is in the radical of the pairing $\langle-,-\rangle_{\chi}$. 
The subspace of $\chi$-negligible morphisms is denoted  $\Nn_{\chi}^{X,Y}\subseteq \C_{\textrm{univ}}^{\T}(X,Y)$, 
and we write $\Nn_{\chi}$ for the tensor ideal of all $\Nn_{\chi}^{X,Y}$.
\end{definition}

\begin{definition}
The category $\C_{\chi}$ is the Karoubian envelope of the quotient $\C_{\textrm{univ}}^{\T}/\Nn_{\chi}$.
\end{definition}

\begin{proposition}
The category $\C_{\chi}$ is tensor generated by a rigid object $W$ that comes with the structure of a closed $\Lambda_r$-Frobenius algebra, and $\End_{\C_{\chi}}(\one) \cong \k$. 
If $\Sigma$ is a closed surface with an $r$-spin structure, then the invariant that $W$ assigns to $\Sigma$ is exactly $\chi(\Sigma)$.
\end{proposition}

\begin{proof}
The category $\C_{\textrm{univ}}^{\T}$ is already tensor generated by a rigid object $W$ that is a closed $\Lambda_r$-Frobenius algebra. This property descents directly to the quotient $\C_{\chi}$. The collection of $\chi$-negligible morphisms in $\End_{\C_{\textrm{univ}}^{\T}}(\one) = \winf^{\T}$ is exactly $\Ker(\chi)$ so we get that $\End_{\C_{\chi}}(\one) = \winf^{\T}/\Ker(\chi)\cong \k$. Under the last isomorphism, the invariant that is associated to $\Sigma$ is exactly the image of $\Sigma\in \winf^{\T}$ inside $\winf^{\T}/\Ker(\chi)$, which is $\chi(\Sigma)$. 
\end{proof}

This answers the question raised in Section~\ref{sec:Introduction}. 
For every set of invariants we have constructed here a category with a closed $\Lambda_r$-Frobenius algebra that assigns exactly these values. 
However, the categories $\C_{\chi}$ are nicely behaved only for a specific collection of characters:

\begin{definition}
A category $\C$ is called \textsl{$\k$-good} if it is rigid, abelian, enriched over finite-dimensional $\k$-vector spaces, and $\End_{\C}(\one)\cong\k$. 
\end{definition}

\begin{definition}
A character $\chi\colon\winf^{\T}\to \k$ is called \textsl{good} if the following conditions are satisfied:
\begin{enumerate}
\item All the Hom spaces in $\C_{\chi}$ are finite-dimensional. 
\item Every nilpotent endomorphism in $\C_{\chi}$ has trace zero. 
\end{enumerate}
\end{definition}

The following is an adaptation of a result in \cite{meir21} to closed $\Lambda_r$-Frobenius algebras.

\begin{theorem}\cite[Thm.\,1.1]{meir21} Let $\chi \colon \winf^{\T}\to \k$ be a character. 
The following conditions are equivalent:
\begin{enumerate}
\item The category $\C_{\chi}$ is a semisimple $\k$-good category.
\item The character $\chi$ is the character of invariants of some closed $\Lambda_r$-Frobenius algebra in some $\k$-good category.
\item The character $\chi$ is a good character. 
\end{enumerate}
\end{theorem}

The above theorem already gives us a necessary condition for~$\chi$ to be good:

\begin{proposition}
Let $\chi\colon\winf^{\T}\to \k$ be a good character. 
If $r$ is odd then 
\begin{equation}
\sum_{n\geq 0} \chi(\alpha_n)X^n
\end{equation}
is a rational function that is contained in the space $R:=\operatorname{span}\{\frac{1}{1-\lambda x}\}_{\lambda\in \k}$. 
If $r$ is even then both 
\begin{equation}
\sum_{n\geq 0} \chi(\alpha_n^+)X^n
	\quad \textrm{and} \quad 
	\sum_{n\geq 0} \chi(\alpha_n^-)X^n
\end{equation}
are contained in $R$.
\end{proposition}

\begin{proof}
For the case $r$ is odd and for the series $\sum_{n\geq 0} \chi(\alpha_n^+)X^n$ when $r$ is even this follows from \cite[Cor.\,2.3]{meir21}, by considering the handle morphisms. In the last series we argue as follows: write $T_1 = (h^+_{x+2})^{\mathcal Z} \circ (h^{+}_{x+4})^{\mathcal Z} \circ \dots \circ (h^{+}_{x-2})^{\mathcal Z} \circ (h^{+}_{x})^{\mathcal Z}$ and $T_2= (h^+_{x+2})^{\mathcal Z} \circ (h^{+}_{x+4})^{\mathcal Z} \circ \dots \circ (h^{+}_{x-2})^{\mathcal Z} \circ (h^{-}_{x})^{\mathcal Z}$. Then the series is actually $$\sum_{n\geq 0} \Tr(T_2^{n-1}T_1)X^n.$$ Lemma \ref{lem:identities} implies that $T_1T_2=T_2T_1$ and $T_1^2=T_2^2.$ We can then decompose the relevant space into eigenspaces for $T_1$ and $T_2$, and the result follows, again by \cite[Cor.\,2.3]{meir21}. 
\end{proof}

\subsection{Example: the constant character and Deligne's category \texorpdfstring{$\Rep(S_t)$}{Rep(S\_t)}}

We first consider the case $r=1$, and let $t\in \k$ be any number. 
Consider the character $\chi_t$ that sends all the closed surfaces to $t$.
In \cite[Sect.\,8]{meir21} it is shown that $\chi_t$ is a good character, and that the resulting category $\C_{\chi_t}$ is the category $\Rep(S_t)$ that was introduced by Deligne in \cite{Del07} as the interpolation of the representation categories of symmetric groups. 
By pulling back along $\Bord_2^r\to \Bord_2^1$ we see that the constant character $\chi_t(\Sigma) = t$ is a good character also for arbitrary~$r$. 
Intuitively, this algebra can be thought of as the commutative algebra $\k^t$ of all functions from the set with $t$ elements to the ground field, with pointwise addition and multiplication. Of course, this makes actual sense only when $t\in\N$. 
The mere existence of the last algebra enables us to prove the following: 

\begin{lemma}
The set of good characters $\winf^{\T}\to \k$ forms a $\k$-algebra.
\end{lemma}

\begin{proof}
Assume that $\chi_1$ and $\chi_2$ are two good characters. 
Then there are categories $\D_1$ and $\D_2$ and closed $\Lambda_r$-Frobenius algebras $W_1$ and $W_2$ that afford these characters. 
Both algebras can be regarded as algebras inside the Deligne product $\D_1\boxtimes \D_2$. 
As such, we can take their sums and products. 
By Lemma \ref{lem:plustimes} we see that pointwise addition and pointwise multiplication of good characters is again a good character, as the category $\D_1\boxtimes \D_2$ is again a good category.
Since we also have an algebra with constant character values, this also shows that the set of good characters is closed under multiplication by a scalar $t\in \k$, and the set of good characters is indeed a $\k$-algebra under pointwise addition and multiplication of characters.
\end{proof}

We are now ready to prove the following:
\begin{theorem}
If $r$ is odd then the character $\chi$ is good if and only if 
\begin{equation}
\sum_{n\in\Z_{\geqslant 1}} \chi(\alpha_n) X^n
\;\in\; 
\operatorname{span}\Big\{\frac{1}{1-\lambda X}\Big\}_{\lambda\in \k} \, . 
\end{equation}
If $r$ is even then the character $\chi$ is good if and only if 
only if 
\begin{equation}
\sum_{n\in\Z_{\geqslant 1}} \chi(\alpha_n^+)X^n\,,\; \sum_{n\in\Z_{\geqslant 1}} \chi(\alpha_n^-)X^n
\;\in\; 
\operatorname{span}\Big\{\frac{1}{1-\lambda X}\Big\}_{\lambda\in \k} \,.
\end{equation}
\end{theorem}

\begin{proof}
We have already seen that this condition is necessary. 
To show that it is sufficient, we use the fact that the set of good characters form a $\k$-algebra.
Consider the case where~$r$ is odd. 
The characters for the algebras $E_{\mu}$, and $B^{(d)}$, with $d|r$, span the set of all characters that satisfies the condition of the theorem. We are done in this case. 
If $r$ is even, we use the characters of the algebras $A$, $A\underline{\ot} E_{\mu}$, and $C^{(d)}$ for $d|r$. 
Again, they span the set of all characters that satisfy the above condition, and therefore we are done in this case as well. 
\end{proof}

\begin{remark}
All the algebras that we constructed in this paper can be realised inside a symmetric monoidal category of the form 
\begin{equation}
  \SVect_\k \boxtimes \bigboxtimes_i \Rep(S_{t_i})
\end{equation}
for some $t_i\in \k$. 
However, if~$\chi$ is a good character, it is possible that the category $\C_{\chi}$ will not be of this form.
\end{remark}

\subsection{Connection to the work of Khovanov, Ostrik, and Kononov}
\label{subsec:KOK}

In case $r=1$ the question of possible invariant values for 2-dimensional TQFTs was studied in \cite{KOK}. 
One gets the same numerical criterion \cite[Thm.\,3.7]{KOK}, where the question of abelian realisation of a given sequence was studied. 
The results in \cite{KOK} also consider the case of a field of positive characteristic. The category $\C_{\chi}$ that we construct here is the same as the category denoted $\underline{\text{DCob}}_{\alpha}$ in \cite{KOK}. 
The category $\text{Cob}_{\alpha}$ appears in \cite{meir21} as the category~$\widetilde{\C}_{\chi}$.


\begin{thebibliography}{CMM}
	
	\bibitem[CMM]{CMM}
	N.~Carqueville and F.~Montiel Montoya, 
	\textsl{Extending Landau-Ginzburg models to the point},
	\href{https://dx.doi.org/10.1007/s00220-020-03871-5}{Comm.\ Math.\ Phys.\ \textbf{379} (2020), 955--977}, 
	\href{https://arxiv.org/abs/1809.10965}{arXiv:1809.10965 [math.QA]}.
	
	\bibitem[Cz]{Czenky}
	A.~Czenky, 
	\textsl{Unoriented 2-dimensional TQFTs and the category $\textrm{Rep}(S_t \wr \mathbb{Z}_2)$}, 
	\newblock in preparation.
	
	\bibitem[De]{Del07} 
	P.~Deligne, 
	\textsl{La cat{\'e}gorie des repr{\'e}sentations du groupe sym{\'e}trique $S_t$, lorsque $t$ n'est pas un entier naturel}, 
	Algebraic groups and homogeneous spaces, Tata Inst. Fund. Res. Stud. Math., vol. 19, Tata Inst. Fund. Res., Mumbai, 2007, pp.\ 209-273

    \bibitem[GG]{Geiges:2012rs}
      H.~Geiges and J.~Gonzalo Pérez, 
      \textsl{Generalised spin structures on 2-dimensional orbifolds}, 
      \newblock 
      \href{https://projecteuclid.org/journals/osaka-journal-of-mathematics/volume-49/issue-2/Generalised-spin-structures-on-2-dimensional-orbifolds/ojm/1340197934.full}{Osaka J. Math. {\bfseries 49} (2012), 449--470}, 
      \href{https://arxiv.org/abs/1004.1979}{arXiv:1004.1979 [math.GT]}.
      
      \bibitem[KOK]{KOK}
      M.~Khovanov, V.~Ostrik, and Y.~Kononov, 
      \textsl{Two-dimensional topological theories, rational functions and their tensor envelopes}, 
      Selecta Mathematica (2022) \textbf{28}:71, 
      \href{https://arxiv.org/abs/2011.14758}{arXiv:2011.14758 [math.QA]}.

      \bibitem[Ko]{Kockbook}
      J.~Kock, 
      \textsl{Frobenius algebras and 2D topological quantum field theories}, 
      \textsl{London Mathematical Society Student Texts} \textbf{59}, Cambridge University Press, 2003.
	
	\bibitem[KR]{kr0401268}
	M.~Khovanov and L.~Rozansky, 
	\textsl{Matrix factorizations and link homology},
	\href{http://journals.impan.gov.pl/fm/Inf/199-1-1.html}{Fund. Math. \textbf{199} (2008), 1--91},
	\href{https://arxiv.org/abs/math/0401268}{arXiv:math/0401268 [math.QA]}.
	
	\bibitem[Me1]{meir20} 
	E.~Meir, 
	\textsl{Universal Rings of Invariants}, 
	International Mathematics Research Notices, Vol. 2022, No. 17, 13128--13180, 
	\href{https://arxiv.org/abs/2007.03845}{arXiv:2007.03845 [math.RT]}.

    \bibitem[{Me2}]{meir21} 
    E.~Meir, 
    \textsl{Interpolations of monoidal categories and algebraic structures by invariant theory}, 
      \href{https://arxiv.org/abs/2105.04622}{arXiv:2105.04622 [math.QA]}.

    \bibitem[{Ra}]{Randal:2014rs}
      O.~{Randal-Williams}, 
      \textsl{{Homology of the moduli spaces and mapping class groups of framed, $r$-Spin and Pin surfaces}}.
      \newblock \href{http://dx.doi.org/10.1112/jtopol/jtt029}{J. Topol. {\bfseries 7} (2014), 155--186}, 
      \href{https://arxiv.org/abs/1001.5366}{arXiv:1001.5366 [math.GT]}.
	
	\bibitem[RSP]{ReutterSchommerPries2022}
	D.~Reutter and C.~Schommer-Pries, 
	\textsl{Semisimple Field Theories Detect Stable Diffeomorphism}, 
	\href{https://arxiv.org/abs/2206.10031}{arXiv:2206.10031 [math.AT]}.
	
    \bibitem[{StSz}]{Szegedy-Stern}
      W.~{Stern} and L.~{Szegedy}, 
      \textsl{Topological field theories on open-closed $r$-spin surfaces}, 
      \newblock 
      \href{https://doi.org/10.1016/j.topol.2022.108062}{Topol. Its Appl. {\bfseries 312} (2022), 108062}, 
      \href{https://arxiv.org/abs/2004.14181}{arXiv:2004.14181 [math.QA]}.

    \bibitem[Sz]{Szegedy:2018phd}
      L.~Szegedy, 
      \textsl{State-sum construction of two-dimensional functorial field theories}, PhD thesis, Universit\"at Hamburg, 2018,
      \url{https://ediss.sub.uni-hamburg.de/handle/ediss/7848}.
      
      \bibitem[TT]{TuraevTurnerUnoriented}
      V.~Turaev and P.~Turner, 
      \textsl{Unoriented topological quantum field theory and link homology}, 
      \href{https://dx.doi.org/10.2140/agt.2006.6.1069}{Algebraic \& Geometric Topology \textbf{6} (2006), 1069--1093},  
      \href{https://arxiv.org/abs/math/0506229}{arXiv:math/0506229 [math.GT]}.

\end{thebibliography}
\end{document}